\documentclass[11pt]{amsart}

\usepackage{fullpage}
\usepackage{times}
\usepackage{amssymb}
\usepackage{amsmath}
\usepackage{amsxtra}
\usepackage{amscd}
\usepackage{graphicx}
\usepackage{amsfonts}
\usepackage{pb-diagram}
\usepackage{float}
\usepackage{enumerate}
\usepackage{dsfont}
\usepackage{xypic}
\numberwithin{equation}{section}

\newtheorem{thm}{Theorem}[section]
\newtheorem{lem}[thm]{Lemma}
\newtheorem{cor}[thm]{Corollary}
\newtheorem{prop}[thm]{Proposition}

\newtheorem{exmp}[thm]{Example}

\newtheorem{rem}[thm]{Remark}

\def\leq{\leqslant}

\def\wo{\overline}
\def\ul{\underline}
\def\wt{\widetilde}
\def\Z{{\mathbb Z}}

\def\C{{\mathbb C}}
\def\N{\mathbb{N}}

\newcommand{\cc}{\mathrm{c}}

\newcommand{\pos}{\mathrm{p}}
\def\cT{\mathcal{T}}
\def\bv{\mathbf{v}}
\def\rv{\mathrm{v}}
\newcommand{\blambda}{\boldsymbol{\lambda}}

\newcommand{\bmu}{\boldsymbol{\mu}}

\newcommand{\btheta}{\boldsymbol{\rho}}

\newcommand{\YH}{{\rm Y}_{d,n}(q)}
\newcommand{\FTL}{{\rm FTL}_{d,n}(q)}
\newcommand{\Irr}{{\rm Irr}}

\title[Representation theory and an isomorphism theorem for the Framisation of the Temperley--Lieb algebra]
  {Representation theory and an isomorphism theorem for  the Framisation of the Temperley--Lieb algebra}
\author{Maria Chlouveraki}
\author{Guillaume Pouchin}

\keywords{Temperley--Lieb algebra, Yokonuma--Hecke algebra, Framisation of the Temperley--Lieb algebra}
\subjclass[2010]{20C08, 05E10, 16S80}
\thanks{We are grateful to  Dimoklis Goundaroulis, Jes\'us Juyumaya, Aristides Kontogeorgis and Sofia Lambropoulou for introducing us to this whole range of problems, and for many fruitful conversations. We would also like to thank Tam\'as Hausel for his interesting questions that led us to provide some extra results on the Yokonuma--Hecke algebra.
Finally, we would like to thank Lo\"ic Poulain d'Andecy and Nicolas Jacon for our useful conversations on the second part of this paper. 
This research has been co-financed by the European Union (European
Social Fund - ESF) and Greek national funds through the Operational Program
``Education and Lifelong Learning" of the National Strategic Reference
Framework (NSRF) - Research Funding Program: THALIS. The second author gratefully acknowledges  financial support of EPSRC through the  grant  Ep/I02610x/1.}

\begin{document}

\maketitle

\begin{abstract}
In this paper, we describe the irreducible representations and give a dimension formula for the Framisation of the Temperley--Lieb algebra. We then prove that the Framisation of the Temperley--Lieb algebra is isomorphic to a direct sum of matrix algebras over tensor products of classical Temperley--Lieb algebras. This allows us to construct a basis for it. We also study in a similar way the Complex Reflection Temperley--Lieb algebra.
\end{abstract}

\section{Introduction}

The Temperley--Lieb algebra was introduced by Temperley and Lieb in \cite{TL} for its applications in statistical mechanics. It was later shown by Jones \cite{jo1,jo} that it can be obtained as a quotient of the Iwahori--Hecke algebra of type $A$. Both algebras depend on a parameter $q$. Jones showed that there exists a unique Markov trace, called the Ocneanu trace, on the Iwahori--Hecke algebra, which depends on a parameter $z$. For a specific value of $z$, the Ocneanu trace passes to the Temperley--Lieb algebra.
Jones used the Ocneanu trace on the Temperley--Lieb algebra to define a polynomial knot invariant, the famous Jones polynomial. Using  the Ocneanu trace as defined  originally on the Iwahori--Hecke algebra of type $A$ yields another famous polynomial invariant, the HOMFLYPT polynomial, which is also known as the $2$-variable Jones polynomial (the $2$ variables being $q$ and $z$).

Yokonuma--Hecke algebras were introduced by Yokonuma \cite{yo} as generalisations of Iwahori--Hecke algebras in the context of finite Chevalley groups. The Yokonuma--Hecke algebra of type $A$ is the centraliser algebra associated to the permutation representation  with respect to a maximal unipotent subgroup of the general linear  group over a finite field. 
Juyumaya has given a generic presentation for this algebra, depending on a parameter $q$, and defined a Markov trace on it, the latter depending on several parameters \cite{ju,ju2,ju3}. This trace was subsequently used by Juyumaya and Lambropoulou for the construction of invariants for framed knots and links \cite{jula1,jula2}. They later showed that these invariants can be also adapted for classical and singular knots and links \cite{jula3, jula4}. The next step was to construct an analogue of the Temperley--Lieb algebra in this case.

As it is explained in more detail in \cite{jula5}, where the technique of framisation is thoroughly discussed, three possible candidates arose. 
The first candidate was the Yokonuma--Temperley--Lieb algebra, which was defined in \cite{gjkl1} 
as the quotient of the Yokonuma--Hecke algebra by exactly the same ideal as the one used by Jones in the classical case. We studied the representation theory of this algebra and constructed a basis for it in \cite{CP}. The values of the parameters for which Juyumaya's Markov trace passes to the Yokonuma--Temperley--Lieb algebra are given in \cite{gjkl1}. 
For these values, the invariants for classical knots and links obtained from the Yokonuma--Temperley--Lieb algebra are equivalent to the Jones polynomial.

A second candidate, which is more interesting knot theoretically, was suggested in \cite{gjkl2}. This is the Framisation of the Temperley--Lieb algebra, which we study in this paper. The Framisation of the Temperley--Lieb algebra is defined in  a subtler way than the Yokonuma--Temperley--Lieb algebra, as the quotient of the Yokonuma--Hecke algebra by a more elaborate ideal, and it is larger than the Yokonuma--Temperley--Lieb algebra. The values of the parameters for which Juyumaya's Markov trace passes to this quotient are given in \cite{gjkl2}. It was recently shown that the invariants for classical links obtained from the Yokonuma--Hecke algebra are stronger than the HOMFLYPT polynomial \cite{CJKL}. It turns out that, in a similar way,  the invariants for classical links obtained from the Framisation of the Temperley--Lieb algebra are stronger than the Jones polynomial.

The third candidate is the  Complex Reflection Temperley--Lieb algebra, defined also in \cite{gjkl2}, which is larger than the Framisation of the Temperley--Lieb algebra, but provides the same knot theoretical information.

In the first part of this paper, we study the representation theory of the Framisation of the Temperley--Lieb algebra. In Theorem \ref{res1} we give a complete description of its irreducible representations, by showing which irreducible representations of the Yokonuma--Hecke algebra pass to the quotient. The representations  of the Yokonuma--Hecke algebra of type $A$ were first studied by Thiem \cite{Thi1,Thi2,Thi3}, but here we use their explicit  description given later in \cite{ChPo}. Our result generalises in a natural way the analogous result in the classical case. We then use the dimensions of the irreducible representations of the Framisation of the Temperley--Lieb algebra in order to compute the dimension of the algebra. We deduce a combinatorial formula involving Catalan numbers, given in Theorem \ref{res2}.

We also take this opportunity to write down the relations between three types of generators used in the literature so far, and show that the Yokonuma--Hecke algebra is split semisimple over a smaller field than the one considered in \cite{ChPo}.

In the second part of this paper, we provide an algebraic connection between the Framisation of the Temperley--Lieb algebra and the Temperley--Lieb algebra. Lusztig \cite{Lu} has shown that Yokonuma--Hecke algebras 
 are isomorphic to direct sums of matrix algebras over certain subalgebras of classical Iwahori--Hecke algebras. For the Yokonuma--Hecke algebra of type $A$, these are all  Iwahori--Hecke algebras  of type $A$. This result was reproved recently in \cite{JP} using Juyumaya's presentation for the Yokonuma--Hecke algebra of type $A$.  Another proof of the same result has been given recently in \cite{ER}. In Theorem \ref{our iso}, we show that the isomorphism can be defined over a smaller ring than the one considered in all three papers. Then, in Theorem \ref{thm-iso-CP}, we prove that the Framisation of the Temperley--Lieb algebra is isomorphic to a direct sum of matrix algebras over tensor products of Temperley--Lieb algebras. Using this result, we provide a basis for the Framisation of the Temperley--Lieb algebra in Proposition \ref{FTLbasis}.

We would like to remark that these isomorphism theorems discussed above render the fact that the invariants for classical links arising from the Yokonuma--Hecke algebra and the Framisation of the Temperley--Lieb algebra are stronger than the HOMFLYPT and the Jones polynomial respectively even more surprising and intriguing.

Finally, in the last section, we study the representation theory and give a dimension formula for the Complex Reflection Temperley--Lieb algebra. We then prove that the Complex Reflection Temperley--Lieb algebra is isomorphic to  a direct sum of matrix algebras over tensor products of Temperley--Lieb and Iwahori--Hecke algebras. We also give  a basis for it. Our results in this section, combined with the results on the Famisation of the Temperley--Lieb algebra in the previous sections and on the Yokonuma--Temperley--Lieb algebra in \cite{CP}, provide a clear indication that the Framisation of the Temperley--Lieb algebra is the natural analogue of the Temperley--Lieb algebra in the context of Yokonuma--Hecke algebras.

\section{Representation theory of the Temperley--Lieb algebra}

In this section, we recall the definition of the Temperley--Lieb algebra as a quotient of the Iwahori--Hecke algebra of type $A$ given by Jones \cite{jo}, and some classical results on its representation theory.

\subsection{The Iwahori--Hecke algebra $\mathcal{H}_n(q)$}
Let  $n \in \N$ and let $q $ be an indeterminate.  The \emph{Iwahori--Hecke algebra of type} $A$, denoted by $\mathcal{H}_n(q)$, is a $\C[q,q^{-1}]$-associative algebra generated by the elements
$$
 G_1, \ldots, G_{n-1}
$$
subject to the following braid relations:
\begin{equation}\label{modularh}
\begin{array}{rcll}
 G_iG_j & = & G_jG_i & \mbox{for all  $i,j=1,\ldots,n-1$ with $\vert i-j\vert > 1$,}\\
 G_iG_{i+1}G_i & = & G_{i+1}G_iG_{i+1} & \mbox{for  all  $i=1,\ldots,n-2$,}\\
\end{array}
\end{equation}
 together with the quadratic relations:
\begin{equation}\label{quadrh}
G_i^2 =q+ (q-1) G_i\qquad \mbox{for  all $i=1,\ldots,n-1$.}
\end{equation}

\begin{rem}{\rm
If we specialise $q$ to $1$, the defining relations (\ref{modularh})--(\ref{quadrh}) become the defining relations for the symmetric group $\mathfrak{S}_n$. Thus, the algebra $\mathcal{H}_n(q)$ is a deformation of $\C[\mathfrak{S}_n]$, the group algebra of $\mathfrak{S}_n$ over $\C$.}
\end{rem}

Let $w \in \mathfrak{S}_n$ and let $w=s_{i_1}s_{i_2}\ldots s_{i_r}$ be a reduced expression for $w$, where $s_i$ denotes the transposition $(i,i+1)$. By Matsumoto's lemma, the element $G_w:=G_{i_1}G_{i_2}\ldots G_{i_r}$ is well defined. It is well-known that the set 
$\{G_w\}_{w \in \mathfrak{S}_n}$ forms a basis of  $\mathcal{H}_n(q)$ over $\C[q,q^{-1}]$, which is called the \emph{standard basis}. In particular, $\mathcal{H}_n(q)$
is a free $\C[q,q^{-1}]$-module of rank $n!$.

\subsection{The  Temperley--Lieb algebra ${\rm TL}_{n}(q)$} 
Let $i=1,\ldots,n-2$. We set
$$G_{i,i+1}:=1 + G_i + G_{i+1} + G_iG_{i+1} +G_{i+1}G_i + G_iG_{i+1}G_i =\sum_{w \in {\langle s_i,s_{i+1}\rangle}}G_w.$$
We define the  \emph{Temperley--Lieb algebra} ${\rm TL}_{n}(q)$ to be the quotient 
$\mathcal{H}_n(q)/I_n$, where $I_n$ is the ideal generated by the element $G_{1,2}$ (if $n \leq 2$, we take $I_n = \{0\}$).
We have
$ G_{i,i+1} \in I_n$  for all $i=1,\ldots,n-2$, since
$$G_{i,i+1} = (G_1G_2 \ldots G_{n-1})^{i-1}\, G_{1,2}\, (G_1G_2 \ldots G_{n-1})^{-(i-1)}.$$

\subsection{Combinatorics of partitions}
Let $\lambda\vdash n$ be a partition of $n$, that is, $\lambda=(\lambda_1,\dots,\lambda_k)$ is a family of  positive integers such that $\lambda_1\geq\lambda_2\geq\dots\geq\lambda_k \geq 1$ and $|\lambda|:=\lambda_1+\dots+\lambda_k=n$. We also say that $\lambda$ is a partition {\em of size} $n$. 

We identify partitions with their Young diagrams: the Young diagram of $\lambda$ is a left-justified array of $k$ rows such that
the $j$-th row contains  $\lambda_j$ {\em nodes } for all $j=1,\dots,k$. We write $\rho=(x,y)$ for the node in row $x$ and column $y$. 

For a node $\rho$ lying in the line $x$ and the column $y$ of $\lambda$ (that is, $\rho=(x,y)$), we define $\cc(\rho):=q^{y-x}$. The number  $\cc(\rho)$ is called the \emph{(quantum) content} of $\rho$. 

Now, a  {\em tableau of shape $\lambda$} is a bijection between the set $\{1,\dots,n\}$ and the set of nodes in $\lambda$. In other words, a tableau of shape $\lambda$ is obtained by placing the numbers $1,\dots,n$ in the nodes of $\lambda$. 
The \emph{size} of a tableau of shape $\lambda$ is $n$, that is, the size of $\lambda$.  
 A tableau is {\em standard} if its entries  increase along each row and down 
 each column of the Young diagram of $\lambda$.

For a tableau ${\mathcal{T}}$, we denote by $\cc({\mathcal{T}}|i)$  the quantum content  of the node with the number $i$ in it. For example, for the standard tableau ${\mathcal{T}}^{^{\phantom{A}}}\!\!\!\!={\textrm{$
\,\fbox{\scriptsize{$1$}}\fbox{\scriptsize{$2$}}\fbox{\scriptsize{$3$}}\,$}}$ of size $3$, we have
\[\cc({\mathcal{T}}|1)=1\,,\ \ \cc({\mathcal{T}}|2)=q\,\ \ \text{and} \, \ \ \cc({\mathcal{T}}|3)=q^2\,.\]

For any tableau $\cT$ of size $n$ and any permutation $\sigma \in \mathfrak{S}_n$, we denote by $\cT^{\sigma}$ the tableau obtained from $\cT$ by applying the permutation $\sigma$ on the numbers contained in the nodes of $\cT$. We have 
$$
\cc(\cT^{\sigma}|i)=\cc\bigl(\cT|\sigma^{-1}(i)\bigr)\ \ \ \ \ \ \text{for all $i=1,\dots,n$.}
$$
Note that if the tableau $\cT$ is standard, the tableau $\cT^{\sigma}$ is not necessarily standard. 

\subsection{Formulas for the irreducible representations of $\C(q)\mathcal{H}_n(q)$}
 We set $\C(q)\mathcal{H}_n(q):=\C(q) \otimes_{\C[q,q^{-1}]}\mathcal{H}_n(q)$. 
 Let  $\mathcal{P}(n)$ be the set of all partitions of $n$, and let
 $\lambda \in \mathcal{P}(n)$. Let $V_{\lambda}$ be a  $\C(q)$-vector space with a basis $\{\rv_{_{\cT}}\}$ indexed by the standard tableaux of shape $\lambda$. We set $\rv_{_{\cT}}:=0$ for any non-standard tableau $\cT$ of shape $\lambda$. We have the following result on the representations of $\C(q)\mathcal{H}_n(q)$, established in \cite{Hoe}:

\begin{thm}\label{thm-rep-hecke} 
Let $\cT$ be a standard tableau of shape $\lambda  \in \mathcal{P}(n)$. For brevity, we set  $\cc_i:=\cc(\cT|i)$ for $i=1,\dots,n$. 
The vector space $V_{\lambda}$ is an irreducible representation of $\C(q)\mathcal{H}_n(q)$ with the action of the generators on the basis element $\rv_{_{\cT}}$ defined as follows:
for $i=1,\dots,n-1$, 
\begin{equation}\label{rep-G}
G_i(\rv_{_{\cT}})=\frac{q\cc_{i+1}-\cc_{i+1}}{\cc_{i+1}-\cc_i}\,\rv_{_{\cT}}+\frac{q\cc_{i+1}-\cc_i}{\cc_{i+1}-\cc_i}\,\rv_{_{\cT^{s_i}}}\ ,
\end{equation}
where $s_i$ is the transposition $(i,i+1)$.
Further, the set $\{V_{\lambda}\}_{\lambda  \in \mathcal{P}(n)}$ is a complete set of pairwise non-isomorphic irreducible representations of $\C(q)\mathcal{H}_n(q)$.
\end{thm}

\begin{cor}
The algebra $\C(q)\mathcal{H}_n(q)$ is split semisimple. 
\end{cor}

\subsection{Irreducible representations of $\C(q){\rm TL}_n(q)$}
Since the algebra $\C(q)\mathcal{H}_n(q)$ is semisimple, the algebra $\C(q){\rm TL}_{n}(q):=\C(q) \otimes_{\C[q,q^{-1}]}{\rm TL}_{n}(q)$ is also semisimple. Moreover, we  have that the irreducible representations of $\C(q){\rm TL}_{n}(q)$ are 
precisely the irreducible representations of $\C(q)\mathcal{H}_n(q)$ that pass to the quotient. That is, 
$V_{\lambda}$ is an irreducible representation of $\C(q){\rm TL}_{n}(q)$ if and only if
$G_{1,2} (\rv_{_{\cT}}) =0$ for every standard tableau ${\cT}$ of shape $\lambda$. 
It is easy to see that the latter is equivalent to the trivial representation not being a direct summand of the restriction $\mathrm{Res}_{\langle s_1,s_2\rangle}^{\mathfrak{S}_n}(E^{\lambda})$, where $E^{\lambda}$ is the irreducible representation of the symmetric group $\mathfrak{S}_n$ (equivalently, the algebra $\C\mathcal{H}_n(1)$) labelled by $\lambda$. We obtain the following description of the irreducible representations of $\C(q){\rm TL}_n(q)$:

\begin{prop}\label{classical case}
We have that $V_{\lambda}$ is an irreducible representation of $\C(q){\rm TL}_n(q)$ if and only if the Young diagram of $\lambda$ has at most two columns.
\end{prop}

\subsection{The dimension of $\C(q){\rm TL}_{n}(q)$}

For $n \in \N$, we denote by $C_n$ the $n$-th Catalan number, that is, the number
$$C_n = \frac{1}{n+1} \binom{2n}{n} =\frac{1}{n+1} \sum_{k=0}^n \binom{n}{k}^2 .$$
We have the following standard result on the dimension of $\C(q){\rm TL}_{n}(q)$:

\begin{prop}\label{dim-TL} 
We have
$${\rm dim}_{\C(q)}(\C(q){\rm TL}_{n}(q)) = C_n.$$
\end{prop}

\section{Representation theory of the framisation of the Temperley--Lieb algebra}

In this section, we  look at a generalisation of the Temperley--Lieb algebra, which is obtained as a quotient of the Yokonuma--Hecke algebra of type $A$. This algebra was introduced in \cite{gjkl2}, where some of its topological properties were studied. Here we  determine its irreducible representations and calculate its dimension.

\subsection{The Yokonuma--Hecke algebra ${\rm Y}_{d,n}(q)$}
Let  $d,\,n \in \N$. Let $q $ be an indeterminate.  The \emph{Yokonuma--Hecke algebra of type} $A$, denoted by ${\rm Y}_{d,n}(q)$, is a $\C[q,q^{-1}]$-associative algebra generated by the elements
$$
 g_1, \ldots, g_{n-1}, t_1, \ldots, t_n
$$
subject to the following relations:
\begin{equation}\label{modular}
\begin{array}{crcll}
\mathrm{(b}_1) & g_ig_j & = & g_jg_i & \mbox{for all  $i,j=1,\ldots,n-1$ with $\vert i-j\vert > 1$,}\\
\mathrm{(b}_2) & g_ig_{i+1}g_i & = & g_{i+1}g_ig_{i+1} & \mbox{for all    $i=1,\ldots,n-2$,}\\
\mathrm{(f}_1) & t_i t_j & =  &  t_j t_i &  \mbox{for all  $i,j=1,\ldots,n$,}\\
\mathrm{(f}_2) & t_j g_i & = & g_i t_{s_i(j)} & \mbox{for  all  $i=1,\ldots,n-1$ and $j=1,\ldots,n$,}\\
\mathrm{(f}_3) & t_j^d   & =  &  1 & \mbox{for  all  $j=1,\ldots,n$,}
\end{array}
\end{equation}
where $s_i$ denotes the transposition $(i, i+1)$, together with the quadratic
relations:
\begin{equation}\label{quadr}
g_i^2 =q + (q-1) \, e_{i} \, g_i\qquad \mbox{for  all $i=1,\ldots,n-1$,}
\end{equation}
 where
\begin{equation}\label{ei}
e_i :=\frac{1}{d}\sum_{s=0}^{d-1}t_i^s t_{i+1}^{-s}.
\end{equation}

Note that we have $e_i^2=e_i$ and $e_ig_i=g_ie_i$ for  all $i=1,\ldots,n-1$.
Moreover, we have 
\begin{equation}\label{eiti}
t_{i}e_i=t_{i+1}e_i\  \,\,\,\,\,\,\mbox{ for all $ i =1,\dots,n-1$.}
\end{equation}

\begin{rem}\label{ref rem}{\rm  
In \cite{ChPo}, the first author and Poulain d'Andecy consider the braid generators 
$\wt{g}_i:=q^{-1/2} g_i$  which satisfy the quadratic relation
\begin{equation}
\wt{g}_i^2 = 1 + (q^{1/2} - q^{-1/2}) e_i \wt{g}_i\ .
\end{equation}
On the other hand, in all the papers  \cite{ju3, jula2, jula3, jula4, chla, gjkl1, gjkl2} prior to \cite{ChPo}, the authors consider the braid generators  
$\wo{g}_i: = \wt{g}_i+(q^{1/2}-1)\, e_{i} \wt{g}_i$ (and thus, $\wt{g}_i:= \wo{g}_i+(q^{-1/2}-1)\, e_{i}  \wo{g}_i$)
which  satisfy the quadratic relation
\begin{equation}
\wo{g}_i^2 =1 +  (q-1) \,e_i + (q-1) \, e_{i} \, \wo{g}_i\ .
\end{equation}
We have $\wo{g}_i = q^{-1/2} g_i + (1-q^{-1/2}) e_i g_i$. Note that
\begin{equation}\label{FTL is ok}
e_ig_i = e_i \wo{g}_i = q^{1/2} e_i \wt{g}_i  \,\,\,\,\,\,\mbox{ for all $ i =1,\dots,n-1$.}
\end{equation}
 }
\end{rem}

\begin{rem}{\rm
If we specialise $q$ to $1$, the defining relations (\ref{modular})--(\ref{quadr}) become the defining relations for the complex reflection group $G(d,1,n) \cong (\Z/d\Z) \wr \mathfrak{S}_n$. Thus, the algebra ${\rm Y}_{d,n}(q)$ is a deformation of $\C[G(d,1,n)]$. Moreover, for $d=1$, the Yokonuma--Hecke algebra ${\rm Y}_{1,n}(q)$ coincides with the Iwahori--Hecke algebra $\mathcal{H}_n(q)$ of type $A$.}
\end{rem}

\begin{rem}{\rm
The relations $({\rm b}_1)$, $({\rm b}_2)$, $({\rm f}_1)$ and $({\rm f}_2)$ are defining relations for the classical framed braid group $\mathcal{F}_n \cong \Z\wr B_n$, where $B_n$ is the classical braid group on $n$ strands, with the $t_j$'s being interpreted as the ``elementary framings" (framing 1 on the $j$th strand). The relations $t_j^d = 1$ mean that the framing of each braid strand is regarded modulo~$d$. Thus, the algebra ${\rm Y}_{d,n}(q)$ arises naturally  as a quotient of the framed braid group algebra over the modular relations ~$\mathrm{(f}_3)$ and the quadratic relations~(\ref{quadr}). Moreover, relations (\ref{modular}) are defining relations for the  modular framed braid group $\mathcal{F}_{d,n}\cong (\Z/d\Z) \wr B_n$, so the algebra ${\rm Y}_{d,n}(q)$ can be also seen  as a quotient of the modular framed braid group algebra over the  quadratic relations~(\ref{quadr}). 
}\end{rem}

Let $w \in \mathfrak{S}_n$ and let $w=s_{i_1}s_{i_2}\ldots s_{i_r}$ be a reduced expression for $w$. By Matsumoto's lemma, the element $g_w:=g_{i_1}g_{i_2}\ldots g_{i_r}$ is well defined. 
Juyumaya \cite{ju3}  has shown that the set 
$$\{t_1^{a_1}t_2^{a_2}\ldots t_n^{a_n} g_w \,|\, 0\leq a_1,a_2,\ldots,a_n \leq d-1,\,w \in \mathfrak{S}_n\}$$ 
forms a basis of  $\YH$ over $\C[q,q^{-1}]$, which is called the \emph{standard basis}. In particular, $\YH$
is a free $\C[q,q^{-1}]$-module of rank $d^nn!$.

\subsection{The Framisation of the Temperley--Lieb algebra ${\rm FTL}_{d,n}(q)$} 
Let $i=1,\ldots,n-2$. We set
$$g_{i,i+1}:=1 + g_i + g_{i+1} + g_ig_{i+1} +g_{i+1}g_i + g_ig_{i+1}g_i = \sum_{w \in {\langle s_i,s_{i+1}\rangle}}g_w.$$
We define the \emph{Framisation of the Temperley--Lieb algebra} to be the quotient 
${\rm Y}_{d,n}(q)/I_{d,n}$, where $I_{d,n}$ is the ideal generated by the element
$e_1e_2\, g_{1,2}$  (if $n \leq 2$, we take $I_{d,n} = \{0\}$). 
Note that, due to (\ref{eiti}), the product $e_1e_2$ commutes with $g_1$ and with $g_2$, so it commutes with $g_{1,2}$. Further, we have
$ e_ie_{i+1}g_{i,i+1} \in I_{d,n}$  for all $i=1,\ldots,n-2$, since
$$e_ie_{i+1}g_{i,i+1} = (g_1g_2 \ldots g_{n-1})^{i-1} \,e_1e_2 \,g_{1,2}\, (g_1g_2 \ldots g_{n-1})^{-(i-1)}.$$

\begin{rem}{\rm In \cite{gjkl2}, the Framisation of the Temperley--Lieb algebra is defined to be the quotient 
${\rm Y}_{d,n}(q)/J_{d,n}$, where $J_{d,n}$ is the ideal generated by the element
$e_1e_2\,\wo{g}_{1,2}$, where
$$\wo{g}_{1,2} = 1 + \wo{g}_1+ \wo{g}_{2} + \wo{g}_1\wo{g}_{2} +\wo{g}_{2}\wo{g}_1 + \wo{g}_1\wo{g}_{2}\wo{g}_1.$$
Due to (\ref{FTL is ok}) and the fact that the $e_i$'s are idempotents, we have $e_1e_2\,\wo{g}_{1,2} = e_1e_2\,g_{1,2}$, and so $I_{d,n}=J_{d,n}$.
}
\end{rem}

\begin{rem}{\rm The ideal $I_{d,n}$ is the ideal generated by the element
$ \sum_{0\leq a,b \leq d-1} t_1^a t_2^b t_3^{-a-b} \,g_{1,2}$.}
\end{rem}

\begin{rem}{\rm
For $d=1$, the Framisation of the Temperley--Lieb algebra ${\rm FTL}_{1,n}(q)$ coincides with the classical Temperley--Lieb algebra $\mathrm{TL}_n(q)$.}
\end{rem}

\subsection{Combinatorics of $d$-partitions}

A $d$-partition $\blambda$ of size $n$ is a $d$-tuple of partitions such that the total number of nodes in the associated Young diagrams is equal to $n$. That is, we have $\blambda=(\blambda^{(1)},\dots,\blambda^{(d)})$ with $\blambda^{(1)},\dots,\blambda^{(d)}$ usual partitions such that $|\blambda^{(1)}|+\dots+|\blambda^{(d)}|=n$.

We write $\btheta=(x,y,k)$ for the node in row $x$ and column $y$ of the Young diagram of $\blambda^{(k)}$, and we say that
$\btheta$ is a \emph{$d$-node} of $\blambda$.  For a $d$-node $\btheta=(x,y,k)$, we define $\pos(\btheta):=k$ and $\cc(\btheta):=q^{y-x}$. The number $\pos(\btheta)$ is the position of $\btheta$ and the number $\cc(\btheta)$ is called the \emph{(quantum) content} of $\btheta$.

Let $\blambda=(\blambda^{(1)},\ldots,\blambda^{(d)})$ be a $d$-partition of $n$. A {\em $d$-tableau of shape $\blambda$} is a bijection between the set $\{1,\dots,n\}$ and the set of $d$-nodes in $\blambda$. In other words, a $d$-tableau of shape $\blambda$ is obtained by placing the numbers $1,\dots,n$ in the $d$-nodes of $\blambda$. 
The \emph{size} of a $d$-tableau of shape $\blambda$ is $n$, that is, the size of $\blambda$.  
 A $d$-tableau is {\em standard} if its entries  increase along each row and down 
 each column of every diagram in $\blambda$. For $d=1$, a standard $1$-tableau is a usual standard tableau.

For a $d$-tableau ${\mathcal{T}}$, we denote respectively by $\pos({\mathcal{T}}|i)$ and $\cc({\mathcal{T}}|i)$  the position and the quantum content  of the $d$-node with the number $i$ in it. For example, for the standard $3$-tableau ${\mathcal{T}}^{^{\phantom{A}}}\!\!\!\!={\textrm{$\left(
\,\fbox{\scriptsize{$2$}}\fbox{\scriptsize{$3$}}\, ,\,\varnothing\, ,\,\fbox{\scriptsize{$1$}}\,\right)$}}$ of size $3$, we have
\[\pos({\mathcal{T}}|1)=3\,,\ \ \pos({\mathcal{T}}|2)=1\,,\ \ \pos({\mathcal{T}}|3)=1\ \ \ \ \ \text{and}\ \ \ \ \  \cc({\mathcal{T}}|1)=1\,,\ \ \cc({\mathcal{T}}|2)=1\,,\ \ \cc({\mathcal{T}}|3)=q\,.\]

For any $d$-tableau $\cT$ of size $n$ and any permutation $\sigma \in \mathfrak{S}_n$, we denote by $\cT^{\sigma}$ the $d$-tableau obtained from $\cT$ by applying the permutation $\sigma$ on the numbers contained in the $d$-nodes of $\cT$. We have 
$$
\pos(\cT^{\sigma}|i)=\pos\bigl(\cT|\sigma^{-1}(i)\bigr)\ \ \ \text{and}\ \ \ \cc(\cT^{\sigma}|i)=\cc\bigl(\cT|\sigma^{-1}(i)\bigr)\ \ \ \ \ \ \text{for all $i=1,\dots,n$.}
$$
Note that if the $d$-tableau $\cT$ is standard, the $d$-tableau $\cT^{\sigma}$ is not necessarily standard. 

\subsection{Formulas for the irreducible representations of $\C(q){\rm Y}_{d,n}(q)$}\label{formYH}
The representation theory of ${\rm Y}_{d,n}(q)$  has been first studied by Thiem \cite{Thi1,Thi2,Thi3} and subsequently by the first author and Poulain d'Andecy \cite{ChPo}, who gave a description of its irreducible representations in terms of $d$-partitions and $d$-tableaux.

Let  $\mathcal{P}(d,n)$ be the set of all $d$-partitions of $n$, and let
 $\blambda \in \mathcal{P}(d,n)$. Let $\wt{V}_{\blambda}$ be a  $\C(q^{1/2})$-vector space with a basis $\{\wt{\bv}_{_{\cT}}\}$ indexed by the standard $d$-tableaux of shape $\blambda$.  In \cite[Proposition 5]{ChPo}, the first author and Poulain d'Andecy describe actions of the generators $\wt{g}_i$, for $i=1,\ldots,n-1$, and $t_j$, for $j=1,\ldots,n$, on $\{\wt{\bv}_{_{\cT}}\}$, which make  $\wt{V}_{\blambda}$ into a representation of  ${\rm Y}_{d,n}(q)$ over $\C(q^{1/2})$.
  The matrices describing the action of the generators $t_j$ have complex coefficients, while the ones describing the action of the generators $\wt{g}_i$ have coefficients in $\C(q^{1/2})$. However, the change of basis 
\begin{equation}\label{change of basis}
 \bv_{_{\cT}} := q^{N_{\cT}/2}\, \wt{\bv}_{_{\cT}}, 
 \end{equation}
where $N_{\cT}: = \# \{ i \in \{1,\ldots,n-1\} \,|\, \pos(\cT|i) < \pos(\cT|i+1)\}$, and the change of generators
\begin{equation}\label{change of generators}
 g_i=q^{1/2}\,\wt{g}_i
\end{equation}
yield a description of the action of ${\rm Y}_{d,n}(q)$ on $\wt{V}_{\blambda}$ which is realised over $\C(q)$ (see theorem below).

 Let $V_{\blambda}$ be a  $\C(q)$-vector space with a basis $\{\bv_{_{\cT}}\}$ indexed by the standard $d$-tableaux of shape $\blambda$. We set $\bv_{_{\cT}}:=0$ for any non-standard $d$-tableau $\cT$ of shape $\blambda$. 
Let $\{\xi_1,\dots,\xi_d\}$ be the set of all $d$-th roots of unity (ordered arbitrarily).
We set $\C(q){\rm Y}_{d,n}(q):=\C(q) \otimes_{\C[q,q^{-1}]}{\rm Y}_{d,n}(q)$. The following result is \cite[Proposition 5]{ChPo} and \cite[Theorem 1]{ChPo}, with the change of basis and generators described by (\ref{change of basis}) and (\ref{change of generators}).

\begin{thm}\label{thm-rep} 
Let $\cT$ be a standard $d$-tableau of shape $\blambda  \in \mathcal{P}(d,n)$. For brevity, we set $\pos_i:=\pos(\cT|i)$ and $\cc_i:=\cc(\cT|i)$ for $i=1,\dots,n$. 
The vector space $V_{\blambda}$ is an irreducible representation of $\C(q){\rm Y}_{d,n}(q)$ with the action of the generators on the basis element $\bv_{_{\cT}}$ defined as follows:
for $j=1,\dots,n$,
\begin{equation}\label{rep-t}
t_j(\bv_{_{\cT}})=\xi_{\pos_j}\bv_{_{\cT}}\  ;
\end{equation}
for $i=1,\dots,n-1$, if $\pos_{i}>\pos_{i+1}$ then
\begin{equation}\label{rep-g1}
g_i(\bv_{_{\cT}})=\bv_{_{\cT^{s_i}}}\ ,
\end{equation}
if $\pos_{i} < \pos_{i+1}$ then
\begin{equation}\label{rep-g2}
g_i(\bv_{_{\cT}})=q\,\bv_{_{\cT^{s_i}}}\ ,
\end{equation}
and if $\pos_{i}=\pos_{i+1}$ then
\begin{equation}\label{rep-g3}
g_i(\bv_{_{\cT}})=\frac{q\cc_{i+1}-\cc_{i+1}}{\cc_{i+1}-\cc_i}\,\bv_{_{\cT}}+\frac{q\cc_{i+1}-\cc_i}{\cc_{i+1}-\cc_i}\,\bv_{_{\cT^{s_i}}}\ ,
\end{equation}
where $s_i$ is the transposition $(i,i+1)$.
Further, the set $\{V_{\blambda}\}_{\blambda  \in \mathcal{P}(d,n)}$ is a complete set of pairwise non-isomorphic irreducible representations of $\C(q){\rm Y}_{d,n}(q)$.
\end{thm}

\begin{cor}
The algebra $\C(q){\rm Y}_{d,n}(q)$ is split semisimple.
\end{cor}

\begin{rem}{\rm Note that 
\begin{equation}\label{eirep}
e_i(\bv_{_{\cT}}) = \left\{
\begin{array}{ll}
 \bv_{_{\cT}} & \text{if }  \pos_{i} = \pos_{i+1} ; \\
 0 & \text{if }  \pos_{i} \neq \pos_{i+1}.
 \end{array}\right.
 \end{equation}}
 \end{rem}

\subsection{Irreducible representations of $\C(q){\rm FTL}_{d,n}(q)$}
Since the algebra $\C(q){\rm Y}_{d,n}(q)$ is semisimple, the algebra
$\C(q){\rm FTL}_{d,n}(q):=\C(q) \otimes_{\C[q,q^{-1}]}{\rm FTL}_{d,n}(q)$ is also semisimple. Moreover, we  have that the irreducible representations of $\C(q){\rm FTL}_{d,n}(q)$ are 
precisely the irreducible representations of $\C(q){\rm Y}_{d,n}(q)$ that pass to the quotient. That is, 
$V_{\blambda}$ is an irreducible representation of $\C(q){\rm FTL}_{d,n}(q)$ if and only if
$e_1 e_2 g_{1,2} (\bv_{_{\cT}}) =0$ for every standard $d$-tableau ${\cT}$ of shape $\blambda$. 

\begin{thm}\label{res1}
We have that $V_{\blambda}$ is an irreducible representation of $\C(q){\rm FTL}_{d,n}(q)$ if and only if the Young diagram of 
$\blambda^{(i)}$ has at most two columns for all $i=1,\ldots,d$.
\end{thm}

\begin{proof}
Let us assume first that $V_{\blambda}$ is an irreducible representation of $\C(q){\rm FTL}_{d,n}(q)$ and let $i \in \{1,\ldots,d\}$.
Set $n_i:=|\blambda^{(i)}|$. If $n_i \leq 2$, then $\blambda^{(i)}$ has at most two columns. If $n_i \geq 3$, let us consider all the standard $d$-tableaux
 $\cT = (\cT^{(1)},\ldots,\cT^{(d)})$ of shape $\blambda$ such that
 $$ \pos_{1}=\pos_{2}=\pos_{3} = \cdots= \pos_{n_i} = i.$$
 Then, using the notation of Theorem \ref{thm-rep-hecke} for the Iwahori--Hecke algebra $\mathcal{H}_{n_i}(q)$ and Equation  (\ref{eirep}), we obtain
 $$G_{1,2}(\rv_{_{{\cT}^{({i})}}}) = g_{1,2}(\bv_{_{\cT}})= g_{1,2}e_1 e_2  (\bv_{_{\cT}}) =e_1 e_2 g_{1,2} (\bv_{_{\cT}})=0$$
Since $\cT^{(i)}$ runs over all the standard tableaux of shape $\blambda^{(i)}$,    
Proposition \ref{classical case} yields that $\blambda^{(i)}$ has at most two columns.

Now assume that $\blambda^{(i)}$ has at most two columns for all $i=1,\ldots,d$.
Let $\cT = (\cT^{(1)},\ldots,\cT^{(d)})$ be a standard $d$-tableau of shape $\blambda$.
If  $\pos_{1}=\pos_{2}=\pos_{3} = :\mathrm{p}$, then, by (\ref{eirep}), $e_1 e_2 g_{1,2} (\bv_{_{\cT}}) =g_{1,2}e_1 e_2  (\bv_{_{\cT}}) = g_{1,2} (\bv_{_{\cT}})$.
In this case,
$g_{1,2}$ acts on $\bv_{_{\cT}}$ in the same way that $G_{1,2}$ acts on $\rv_{_{{\cT}^{(\mathrm{p})}}}$ (replacing the entries greater than $3$ by entries in $\{4,\ldots,|\blambda^{(\rm p)}|\}$).  
Following Proposition \ref{classical case},
we have $g_{1,2} (\bv_{_{\cT}}) = 0$.
Otherwise, again by (\ref{eirep}), we have $e_1e_2  (\bv_{_{\cT}}) =0$, so $e_1 e_2 g_{1,2} (\bv_{_{\cT}}) =g_{1,2}e_1 e_2  (\bv_{_{\cT}})=0$ as desired.
\end{proof}

\subsection{The dimension of $\C(q){\rm FTL}_{d,n}(q)$} We will now use the complete description of the irreducible representations of $\C(q){\rm FTL}_{d,n}(q)$ by
Theorem \ref{res1} to obtain a dimension formula for $\C(q){\rm FTL}_{d,n}(q)$. Set
$${\rm Comp}_d(n):= \{\mu=(\mu_1,\mu_2,\ldots,\mu_d) \in \N^d\,|\, \mu_1+\mu_2+\cdots+\mu_d = n\}.$$

\begin{thm}\label{res2} We have
$${\rm dim}_{\C(q)}(\C(q){\rm FTL}_{d,n}(q)) = \sum_{\mu \in {\rm Comp}_d(n)} \left( \frac{n!}{\mu_1!\mu_2!\ldots \mu_d!}\right)^2 C_{\mu_1}C_{\mu_2} \cdots C_{\mu_d}.$$
\end{thm}

\begin{proof}
Let us denote by $\mathcal{P}^{\leq 2}(d,n)$ the set of $d$-partitions $\blambda$ of $n$ such that the Young diagram of $\blambda^{(i)}$ has at most two columns for all $i=1,\ldots,d$. By Theorem \ref{res1}, and since the algebra $\C(q){\rm FTL}_{d,n}(q)$ is semisimple, 
we have
$${\rm dim}_{\C(q)}(\C(q){\rm FTL}_{d,n}(q)) = \sum_{\blambda \in \mathcal{P}^{\leq 2}(d,n)} \mathrm{dim}_{\C(q)} (V_{\blambda})^2,$$
where $\mathrm{dim}_{\C(q)} (V_{\blambda})$ is the number of standard $d$-tableaux of shape $\blambda$.

Fix $\mu \in {\rm Comp}_d(n)$.
We denote by $\mathcal{P}^{\leq 2}(\mu)$ the set of all $d$-partitions $\blambda$ in $\mathcal{P}^{\leq 2}(d,n)$ such that $|\blambda^{(i)}|=\mu_i$ for all  $i=1,\ldots,d$.
We have
$${\rm dim}_{\C(q)}(\C(q){\rm FTL}_{d,n}(q)) = \sum_{\mu \in {\rm Comp}_d(n)}\sum_{\blambda \in \mathcal{P}^{\leq 2}(\mu)} \mathrm{dim}_{\C(q)} (V_{\blambda})^2.$$

Let $\blambda \in \mathcal{P}^{\leq 2}(\mu)$.
 We have
$$\binom{n}{\mu_1}\binom{n-\mu_1}{\mu_2}\binom{n-\mu_1-\mu_2}{\mu_3}\cdots \binom{n-\mu_1-\mu_2-\cdots - \mu_{d-1}}{\mu_d}=\frac{n!}{\mu_1!\mu_2!\ldots \mu_d!}$$
ways to choose the numbers in $\{1,\ldots,n\}$ that will be placed in the nodes of the Young diagram of $\blambda^{(i)}$ for each $i=1,\ldots,d$. We deduce that
$$\mathrm{dim}_{\C(q)} (V_{\blambda}) = \frac{n!}{\mu_1!\mu_2!\ldots \mu_d!} \prod_{i=1}^d \mathrm{dim}_{\C(q)} (V_{\blambda^{(i)}})\ ,$$
where $V_{\blambda^{(i)}}$ is the irreducible representation of $\C(q){\rm TL}_{\mu_i}(q)$ labelled by $\blambda^{(i)}$.
We thus obtain that  
$$ {\rm dim}_{\C(q)}(\C(q){\rm FTL}_{d,n}(q))=\sum_{\mu \in {\rm Comp}_d(n)}\left( \frac{n!}{\mu_1!\mu_2!\ldots \mu_d!}\right)^2 \sum_{\blambda \in \mathcal{P}^{\leq 2}(\mu)} 
\prod_{i=1}^d \mathrm{dim}_{\C(q)} (V_{\blambda^{(i)}})^2.$$

We now have that 
$$\sum_{\blambda \in \mathcal{P}^{\leq 2}(\mu)} 
\prod_{i=1}^d \mathrm{dim}_{\C(q)} (V_{\blambda^{(i)}})^2$$
is equal to
$$\sum_{\blambda^{(1)} \in \mathcal{P}^{\leq 2}(1,\mu_1)}\sum_{\blambda^{(2)} \in \mathcal{P}^{\leq 2}(1,\mu_2)}\ldots \sum_{\blambda^{(d)} \in \mathcal{P}^{\leq 2}(1,\mu_d)}
\prod_{i=1}^d \mathrm{dim}_{\C(q)} (V_{\blambda^{(i)}})^2,
$$
which in turn is equal to 
$$\prod_{i=1}^d \left( \sum_{\blambda^{(i)} \in \mathcal{P}^{\leq 2}(1,\mu_i)} \mathrm{dim}_{\C(q)} (V_{\blambda^{(i)}})^2 \right).$$
By Proposition \ref{dim-TL}, we have that
$$\sum_{\blambda^{(i)} \in \mathcal{P}^{\leq 2}(1,\mu_i)} \mathrm{dim}_{\C(q)} (V_{\blambda^{(i)}})^2 = {\rm dim}_{\C(q)}(\C(q){\rm TL}_{\mu_i}(q)) =C_{\mu_i}\ ,$$
 for all $i=1,\ldots,d$.
We conclude that
$${\rm dim}_{\C(q)}(\C(q){\rm FTL}_{d,n}(q)) = \sum_{\mu \in {\rm Comp}_d(n)} \left( \frac{n!}{\mu_1!\mu_2!\ldots \mu_d!}\right)^2 C_{\mu_1}C_{\mu_2} \cdots C_{\mu_d}.$$

\end{proof}

\section{An isomorphism theorem for the Framisation of the Temperley--Lieb algebra}\label{sec-iso}

Lusztig has proved that Yokonuma--Hecke algebras are isomorphic to  direct sums of matrix algebras over certain subalgebras of classical Iwahori--Hecke algebras \cite[\S 34]{Lu}. For the Yokonuma--Hecke algebras $\YH$, these are all Iwahori--Hecke algebras of type $A$. This result was reproved in \cite{JP} using the presentation of $\YH$ given by Juyumaya. Since we use the same presentation, we will use the latter exposition of the result in order to prove an analogous statement for $\FTL$. Note that in both cases the result has been obtained over the ring $\C[q^{1/2},q^{-1/2}]$ (with the use of the generators $\wt{g}_i$ defined in Remark \ref{ref rem}).  We will show here that it is still valid over the smaller ring $\C[q,q^{-1}]$.

\subsection{Compositions and Young subgroups}\label{Young}
Let $\mu \in {\rm Comp}_d(n)$, where
$${\rm Comp}_d(n)= \{\mu=(\mu_1,\mu_2,\ldots,\mu_d) \in \N^d\,|\, \mu_1+\mu_2+\cdots+\mu_d = n\}.$$
We say that $\mu$ is a {\em composition of $n$ with $d$ parts}. 
The Young subgroup $\mathfrak{S}_\mu$  of $\mathfrak{S}_n$ is the subgroup $\mathfrak{S}_{\mu_1}\times \mathfrak{S}_{\mu_2} \times \cdots 
\times \mathfrak{S}_{\mu_d}$, where  $\mathfrak{S}_{\mu_1}$ acts on the letters
$\{1,\ldots,\mu_1\}$, $\mathfrak{S}_{\mu_2}$ acts on the letters
$\{\mu_1+1,\ldots,\mu_1+\mu_2\}$, and so on. Thus, $\mathfrak{S}_\mu$ is a parabolic subgroup of $\mathfrak{S}_n$ generated by the transpositions $s_j=(j,j+1)$ with $j \in J^\mu :=\{1,\ldots,n-1\} \setminus \{\mu_1,\mu_1+\mu_2,\ldots,\mu_1+\mu_2+\cdots+\mu_{d-1}\}$.

We  have an Iwahori--Hecke algebra $\mathcal{H}^\mu(q)$ associated with $\mathfrak{S}_\mu$, which is the subalgebra of $\mathcal{H}_n(q)$ generated by $\{G_j\,|\, j \in J^\mu\}$. The algebra $\mathcal{H}^\mu(q)$ is a free $\C[q,q^{-1}]$-module 
with basis $\{G_w \,|\, w \in \mathfrak{S}_\mu\}$, and it is isomorphic to the tensor product (over $\C[q,q^{-1}]$) of Iwahori--Hecke algebras 
$\mathcal{H}_{\mu_1}(q) \otimes \mathcal{H}_{\mu_2}(q) \otimes \cdots \otimes \mathcal{H}_{\mu_d}(q)$
(with $\mathcal{H}_{\mu_i}(q) \cong \C[q,q^{-1}]$ if $\mu_i \leq 1$).

For $i=1,\ldots,d$, we denote by $\rho_i$ the natural surjection $\mathcal{H}_{\mu_i}(q) \twoheadrightarrow 
\mathcal{H}_{\mu_i}(q) /I_{\mu_i} \cong {\rm TL}_{\mu_i}(q)$, where
$I_{\mu_i}$ is the ideal generated by $G_{\mu_1+\cdots+\mu_{i-1}+1,\mu_1+\cdots+\mu_{i-1}+2}$  
if $\mu_i >2$ and $I_{\mu_i} =\{0\}$ if $\mu_i \leq 2$.
We obtain that
$\rho^\mu:=\rho_1 \otimes \rho_2 \otimes \cdots \otimes \rho_d$ is a surjective $\C[q,q^{-1}]$-algebra homomorphism
$\mathcal{H}^\mu(q)\twoheadrightarrow {\rm TL}^{\mu}(q)$, where  ${\rm TL}^{\mu}(q)$ denotes the tensor product of Temperley--Lieb algebras ${\rm TL}_{\mu_1}(q) \otimes {\rm TL}_{\mu_2}(q) \otimes \cdots \otimes {\rm TL}_{\mu_d}(q)$.

\subsection{Two isomorphism theorems for the Yokonuma--Hecke algebra $\YH$}
Let $\{\xi_1,\dots,\xi_d\}$ be the set of all $d$-th roots of unity (ordered arbitrarily).  
Let $\chi$ be an irreducible character of the abelian group $\mathcal{A}_{d,n} \cong (\Z/d\Z)^n$ generated by the elements $t_1,t_2,\ldots,t_n$. 
There exists a  primitive idempotent of $\C[\mathcal{A}_{d,n}]$ associated with $\chi$ defined as 
$$E_\chi := \prod_{j=1}^n \left(\frac{1}{d} \sum_{s=0}^{d-1} \chi(t_j^s) t_j^{-s}\right) = \prod_{j=1}^n \left(\frac{1}{d} \sum_{s=0}^{d-1} \chi(t_j)^s t_j^{-s}\right) .$$
Moreover, we can define a composition $\mu^\chi \in{\rm Comp}_d(n)$ by setting
$$\mu_i^\chi := \# \{ j \in \{1,\ldots,n\}\,|\, \chi(t_j) = \xi_i\} \quad \text{for all } i=1,\ldots,d.$$

Conversely, given  a composition $\mu \in {\rm Comp}_d(n)$, we can consider the subset $\Irr^\mu(\mathcal{A}_{d,n})$ of $\Irr(\mathcal{A}_{d,n})$
defined as
$$\Irr^\mu(\mathcal{A}_{d,n}) := \{ \chi \in \Irr(\mathcal{A}_{d,n})\,|\, \mu^\chi = \mu\}.$$
There is an action of $\mathfrak{S}_n$ on $\Irr^\mu(\mathcal{A}_{d,n})$ given by
$$ w(\chi)(t_j) := \chi(t_{w^{-1}(j)}) \quad \text{ for all } w\in \mathfrak{S}_n, \, j=1,\ldots,n.$$
Let $\chi_1^\mu \in \Irr^\mu(\mathcal{A}_{d,n})$ be the character given by
$$\left\{\begin{array}{ccccccc} 
\chi_1^\mu (t_1) & = & \cdots & = & \chi_1^\mu(t_{\mu_1}) & = & \xi_1 \\
\chi_1^\mu (t_{\mu_1+1}) & = & \cdots & = & \chi_1^\mu(t_{\mu_1+\mu_2}) & = & \xi_2 \\
\chi_1^\mu (t_{\mu_1+\mu_2+1}) & = & \cdots & = & \chi_1^\mu(t_{\mu_1+\mu_2+\mu_3}) & = & \xi_3 \\
\vdots &  \vdots  & \vdots & \vdots  & \vdots  & \vdots &\vdots  \\
\chi_1^\mu (t_{\mu_1+\cdots+\mu_{d-1}+1}) & = & \cdots & = & \chi_1^\mu(t_{n}) & = & \xi_d \\
\end{array}\right.$$
The stabiliser of $\chi_1^\mu$ under the action of $\mathfrak{S}_n$ is the Young subgroup $\mathfrak{S}_\mu$. 
In each left coset in $\mathfrak{S}_n/\mathfrak{S}_\mu$,
we can take a representative of minimal length; such a representative is unique (see, for example, \cite[\S 2.1]{GePf}). 
Let 
$$\{ \pi_{\mu,1}, \pi_{\mu,2},\ldots, \pi_{\mu,m_\mu}\}$$
be this set of distinguished left coset representatives of $\mathfrak{S}_n/\mathfrak{S}_\mu$,
with
$$m_\mu = \frac{n!}{\mu_1!\mu_2!\ldots \mu_d!}\,$$
and the convention that $\pi_{\mu,1}=1$. 
Then, if we set
$$\chi_k^{\mu} := \pi_{\mu,k}(\chi_1^{\mu}) \quad \text{for all } k=1,\ldots,m_\mu,$$
we have
$$ \Irr^\mu(\mathcal{A}_{d,n}) = \{ \chi_1^{\mu}, \chi_2^{\mu},\ldots, \chi_{m_\mu}^{\mu}\}.$$

We now set
$$E_\mu: = \sum_{\chi \in \Irr^\mu(\mathcal{A}_{d,n})} E_\chi = \sum_{k=1}^{m_\mu} E_{\chi_k^\mu} .$$
Since
the set $\{E_\chi \,|\, \chi \in \Irr(\mathcal{A}_{d,n})\}$ forms a complete set of orthogonal idempotents in $\YH$, and
\begin{equation}\label{Echi}
 t_j E_\chi = E_\chi t_j = \chi(t_j)E_\chi \quad \text{and} \quad g_wE_\chi = E_{w(\chi)}g_w
\end{equation}
for all $\chi \in \Irr(\mathcal{A}_{d,n})$, $j=1,\ldots,n$ and $w \in \mathfrak{S}_n$, we have that the set $\{E_\mu \,| \,\mu \in {\rm Comp}_d(n)\}$ forms a complete set of central orthogonal idempotents in $\YH$ (cf.~\cite[\S 2.4]{JP}). In particular, we have the following decomposition of $\YH$ into a direct sum of two-sided ideals:
$$\YH = \bigoplus_{\mu \in {\rm Comp}_d(n)} E_\mu \YH.$$

For the moment, let us consider all the algebras defined over the Laurent polynomial ring $\C[q^{1/2},q^{-1/2}]$ (by extension of scalars).
Let $\ell : \mathfrak{S}_n \rightarrow \N$ denote the length function on $\mathfrak{S}_n$.
We define a  $\C[q^{1/2},q^{-1/2}]$-linear map
$$\widetilde{\Psi}_\mu : E_\mu \YH \rightarrow {\rm Mat}_{m_\mu}(\mathcal{H}^\mu(q))$$
as follows:
 for all $k \in \{1,\ldots,m_\mu\}$ and $w \in \mathfrak{S}_n$,  we set
$$ \widetilde{\Psi}_\mu(E_{\chi_k^\mu}g_w) := q^{\frac{1}{2}(\ell(w)-\ell(\pi_{\mu,k}^{-1}w\pi_{\mu,l}))}G_{\pi_{\mu,k}^{-1}w\pi_{\mu,l}} M_{k,l}  \,, $$ 
where  $l \in \{1,\ldots,m_\mu\}$ is uniquely defined by the relation $w(\chi_l^\mu)=\chi_k^\mu$ and
$M_{k,l}$ is the elementary $m_\mu \times m_\mu$ matrix with $1$ in position $(k,l)$.
Note that $\pi_{\mu,k}^{-1}w\pi_{\mu,l} \in \mathfrak{S}_\mu$.
We also define a $\C[q^{1/2},q^{-1/2}]$-linear map
$$\widetilde{\Phi}_\mu :  {\rm Mat}_{m_\mu}(\mathcal{H}^\mu(q)) \rightarrow E_\mu \YH $$
as follows: for all $k,l \in \{1,\ldots,m_\mu\}$ and $w \in \mathfrak{S}_\mu$,  we set
$$ \widetilde{\Phi}_\mu(G_w M_{k,l}) :=   q^{\frac{1}{2}(\ell(w)-\ell(\pi_{\mu,k}^{-1}w\pi_{\mu,l}))}E_{\chi_k^\mu}g_{\pi_{\mu,k}w\pi_{\mu,l}^{-1}} E_{\chi_l^\mu}.$$ 
Using the generators $\wt{g}_i$ and $\widetilde{G}_i$ defined in Remark \ref{ref rem}, the above maps are equivalent to
$$ \widetilde{\Psi}_\mu(E_{\chi_k^\mu}\wt{g}_w) := \widetilde{G}_{\pi_{\mu,k}^{-1}w\pi_{\mu,l}} M_{k,l} \quad  \text{ and } \quad  \widetilde{\Phi}_\mu(\widetilde{G}_w M_{k,l}) :=   E_{\chi_k^\mu}\wt{g}_{\pi_{\mu,k}w\pi_{\mu,l}^{-1}} E_{\chi_l^\mu}.$$
Then we have the following \cite[Theorem 3.1]{JP}:

\begin{thm} \label{thm-iso-JP}
Let $\mu \in {\rm Comp}_d(n)$.  The linear map $\widetilde{\Psi}_\mu$
is an isomorphism of $\C[q^{1/2},q^{-1/2}]$-algebras with inverse map $\widetilde{\Phi}_\mu$. As a consequence, the map
$$\widetilde{\Psi}_n:= \bigoplus_{\mu \in {\rm Comp}_d(n)}  \widetilde{\Psi}_\mu :  \YH \rightarrow \bigoplus_{\mu \in {\rm Comp}_d(n)}  {\rm Mat}_{m_\mu}(\mathcal{H}^\mu(q))$$
is also an isomorphism of $\C[q^{1/2},q^{-1/2}]$-algebras, with inverse map
$$\widetilde{\Phi}_n:= \bigoplus_{\mu \in {\rm Comp}_d(n)}  \widetilde{\Phi}_\mu :   \bigoplus_{\mu \in {\rm Comp}_d(n)}  {\rm Mat}_{m_\mu}(\mathcal{H}^\mu(q)) \rightarrow  \YH.$$
\end{thm}

We will now show that we can construct similar isomorphisms over the smaller ring $\C[q,q^{-1}]$. In order to do this, we will make use of Deodhar's lemma (see, for example, \cite[Lemma 2.1.2]{GePf}) about the distinguished left coset representatives of $\mathfrak{S}_n/\mathfrak{S}_\mu$.

\begin{lem}\label{Deodhar} {\rm (Deodhar's lemma)} Let $\mu \in {\rm Comp}_d(n)$. For all $k \in \{1,\ldots,m_\mu\}$ and $i= 1,\ldots,n-1$, 
let $l \in \{1,\ldots,m_\mu\}$ be uniquely defined by the relation $s_i(\chi_l^\mu)=\chi_k^\mu$.
We have
$$\pi_{\mu,k}^{-1}s_i \pi_{\mu,l} = \left\{ \begin{array}{cl}
 1 & \text{ if } k\neq l ; \\ & \\
s_j  & \text{ if } k = l,
\end{array}\right.$$
for some $j \in J^\mu$.
\end{lem}

Deodhar's lemma implies that, for all $i=1,\ldots,n-1$, $\widetilde{\Psi}_\mu(E_\mu g_i)$ is a symmetric matrix whose diagonal non-zero coefficients are of the form $G_j$ with $j \in J^\mu$, while all non-diagonal non-zero coefficients are equal to $q^{1/2}$. 
Now, let us consider the diagonal matrix
$$U_\mu := \sum_{k=1}^{m_\mu} q^{\ell(\pi_{\mu,k})/2}M_{k,k}.$$
The coefficients of the matrix $U_\mu\widetilde{\Psi}_\mu(E_\mu g_i)U_\mu^{-1}$ satisfy:
$$(U_\mu\widetilde{\Psi}_\mu(E_\mu g_i)U_\mu^{-1})_{k,l}  = q^{(\ell(\pi_{\mu,k})-\ell(\pi_{\mu,l}))/2}  (\widetilde{\Psi}_\mu(E_\mu g_i))_{k,l}\,,$$
for all $k,l \in \{1,\ldots,m_\mu\}$. Therefore, following the definition of $\widetilde{\Psi}_\mu$ and Deodhar's lemma, the matrix $U_\mu\widetilde{\Psi}_\mu(E_\mu g_i)U_\mu^{-1}$ is a matrix whose diagonal  coefficients are the same as the diagonal coefficients of $\widetilde{\Psi}_\mu(E_\mu g_i)$ (and thus of the form $G_j$ with $j \in J^\mu$), while all non-diagonal non-zero coefficients are equal to either $1$ or $q$.
Moreover, since,  for all $j=1,\ldots,n$,   
$$\widetilde{\Psi}_\mu(E_\mu t_j) = \sum_{k=1}^{m_\mu} \chi_k^\mu(t_j) M_{k,k}$$
is a diagonal matrix, we have $U_\mu\widetilde{\Psi}_\mu(E_\mu t_j)U_\mu^{-1}=\widetilde{\Psi}_\mu(E_\mu t_j)$. We conclude the following:

\begin{thm}\label{our iso}
Let $\mu \in {\rm Comp}_d(n)$.  The map 
$${\Psi}_\mu : E_\mu \YH \rightarrow {\rm Mat}_{m_\mu}(\mathcal{H}^\mu(q))$$ defined by
$$ {\Psi}_\mu(E_\mu a):= U_\mu\widetilde{\Psi}_\mu(E_\mu a)U_\mu^{-1},$$ for all $a \in \YH$, is an isomorphism of $\C[q,q^{-1}]$-algebras.
Its inverse is the map
$$ \Phi_\mu: {\rm Mat}_{m_\mu}(\mathcal{H}^\mu(q)) \rightarrow E_\mu \YH$$ defined by
$$ {\Phi}_\mu(A):= \widetilde{\Phi}_\mu(U_\mu^{-1} A U_\mu),$$ for all $ A \in {\rm Mat}_{m_\mu}(\mathcal{H}^\mu(q))$. As a consequence, the map
$${\Psi}_n:= \bigoplus_{\mu \in {\rm Comp}_d(n)}  {\Psi}_\mu :  \YH \rightarrow \bigoplus_{\mu \in {\rm Comp}_d(n)}  {\rm Mat}_{m_\mu}(\mathcal{H}^\mu(q))$$
is also an isomorphism of $\C[q,q^{-1}]$-algebras, with inverse map
$${\Phi}_n:= \bigoplus_{\mu \in {\rm Comp}_d(n)}  {\Phi}_\mu :   \bigoplus_{\mu \in {\rm Comp}_d(n)}  {\rm Mat}_{m_\mu}(\mathcal{H}^\mu(q)) \rightarrow  \YH.$$
\end{thm}

\begin{rem}
{\rm We believe that, using a similar method, one can prove that Lusztig's isomorphism theorem for Yokonuma--Hecke algebras \cite[\S 34]{Lu}  is valid over the ring $\C[q,q^{-1}]$.}
\end{rem}

\subsection{From $\FTL$ to Temperley--Lieb}
Recall that $\FTL$ is the quotient
${\rm Y}_{d,n}(q)/I_{d,n}$, where $I_{d,n}$ is the ideal generated by the element
$e_1e_2\, g_{1,2}$ (with $I_{d,n}=\{0\}$ if $n \leq 2$). Let $\mu \in {\rm Comp}_d(n)$.  We will study the image of $e_1e_2\, g_{1,2}$ under the isomorphism $\Psi_\mu$.

By \eqref{Echi}, for all $i=1,\ldots,n-1$ and $\chi \in \Irr(\mathcal{A}_{d,n})$, we have
\begin{equation}\label{eiEchi}
e_i E_\chi = E_\chi e_i= \frac{1}{d} \sum_{s=0}^{d-1} \chi(t_i)^s \chi(t_{i+1})^{-s} E_\chi =  
\left\{ \begin{array}{cl}
E_\chi & \text{ if }  \chi(t_i) = \chi(t_{i+1}) ; \\ & \\
0  & \text{ if }  \chi(t_i) \neq \chi(t_{i+1}).
\end{array}\right.
\end{equation}
We deduce that, for all $k=1,\ldots,m_\mu$,
\begin{equation}\label{i-1}
E_{\chi_k^\mu}e_1e_2g_{1,2} = 
\left\{ \begin{array}{cl}
E_{\chi_k^\mu}g_{1,2} & \text{ if }  \chi_k^\mu(t_1)=\chi_k^\mu(t_2)=\chi_k^\mu(t_3); \\ & \\
0  & \text{ otherwise } .
\end{array}\right.
\end{equation}

\begin{prop}\label{imageofpsi}
Let $\mu \in {\rm Comp}_d(n)$ and $k \in \{1,\ldots,m_\mu\}$. We have
$$
\Psi_\mu(E_{\chi_k^\mu}e_1e_2g_{1,2}) = 
\left\{ \begin{array}{cl}
G_{i,i+1} M_{k,k} & \text{ for some } i \in  \{1,\ldots,n-2\}
 \,\text{ if }  \chi_k^\mu(t_1)=\chi_k^\mu(t_2)=\chi_k^\mu(t_3); \\ & \\
0  & \text{ otherwise } .
\end{array}\right.
$$
Thus, $\Psi_\mu(E_\mu e_1e_2g_{1,2})$ is a diagonal matrix in $ {\rm Mat}_{m_\mu}(\mathcal{H}^\mu(q))$
with all non-zero coefficients being of the form $G_{i,i+1}$ for some $i \in  \{1,\ldots,n-2\}$.
\end{prop}

\begin{proof}
If $\chi_k^\mu(t_1)=\chi_k^\mu(t_2)=\chi_k^\mu(t_3)$, then $w(\chi_k^\mu)=\chi_k^\mu$ for all
$w \in {\langle s_1,s_2\rangle} \subseteq \mathfrak{S}_n$,
and so
\begin{equation}\label{almost 1}
 \Psi_\mu(E_{\chi_k^\mu}g_{1,2}) = \sum_{w \in {\langle s_1,s_2\rangle}}\Psi_\mu(E_{\chi_k^\mu}g_w) =  
\sum_{w \in {\langle s_1,s_2\rangle}}U_\mu\widetilde{\Psi}_\mu(E_{\chi_k^\mu}g_w)U_\mu^{-1}=
\sum_{w \in{\langle s_1,s_2\rangle}} G_{\pi_{\mu,k}^{-1}w\pi_{\mu,k}} M_{k,k}  . 
\end{equation}
We will show that there exists $i \in \{1,\ldots,n-2\}$ such that
$$ \sum_{w \in{\langle s_1,s_2\rangle}} G_{\pi_{\mu,k}^{-1}w\pi_{\mu,k}} = G_{i,i+1}.$$

By Lemma \ref{Deodhar}, there exist $i,j \in J^\mu$ such that
$$ \pi_{\mu,k}^{-1}s_1\pi_{\mu,k} =s_i \quad \text{and} \quad \pi_{\mu,k}^{-1}s_2\pi_{\mu,k} =s_j.$$
Consequently,
$ \pi_{\mu,k}^{-1}s_1s_2\pi_{\mu,k} =s_is_j$,  $\pi_{\mu,k}^{-1}s_2s_1\pi_{\mu,k} =s_js_i$ and
$ \pi_{\mu,k}^{-1}s_1s_2s_1\pi_{\mu,k} =s_is_js_i$. Moreover, since $s_1$ and $s_2$ do not commute, $s_i$ and $s_j$ do not commute either, so we must have
$j \in \{i-1,i+1\}$. Hence, if $j=i-1$, then
$$\sum_{w \in{\langle s_1,s_2\rangle}} G_{\pi_{\mu,k}^{-1}w\pi_{\mu,k}} = G_{i-1,i},$$
while if $j=i+1$, then
$$\sum_{w \in{\langle s_1,s_2\rangle}} G_{\pi_{\mu,k}^{-1}w\pi_{\mu,k}} = G_{i,i+1}.$$
We conclude that there exists $i \in \{1,\ldots,n-2\}$ such that
$$ \sum_{w \in{\langle s_1,s_2\rangle}} G_{\pi_{\mu,k}^{-1}w\pi_{\mu,k}} = G_{i,i+1},$$
whence we deduce that
$$
 \Psi_\mu(E_{\chi_k^\mu}g_{1,2}) = G_{i,i+1} M_{k,k}.
$$
Combining this with \eqref{i-1} yields the desired result.
\end{proof}

\begin{exmp}\label{example1}{\rm Let us consider the case $d=2$ and $n=4$. We have 
$$(\mu,m_\mu) \in \{ ((4,0),1), ((3,1),4), ((2,2),6), ((1,3),4), ((0,4),1) \}.$$
Then
$$
\Psi_\mu(E_{\chi_k^\mu}e_1e_2g_{1,2}) = 
\left\{ \begin{array}{cl}
G_{1,2} & \text{ if $\mu=(4,0)$ or $\mu=(0,4)$ }, \\ & \\
G_{1,2}\,M_{1,1}  & \text{ if $\mu=(3,1)$ and $k=1$ }, \\  & \\
G_{2,3}\,M_{4,4}  & \text{ if $\mu=(1,3)$ and $k=4$ }, \\  & \\
0  & \text{ otherwise } ,
\end{array}\right.$$
where we take $\pi_{(1,3),4}=s_3s_2s_1$. }
\end{exmp}

Now, recall the surjective   $\C[q,q^{-1}]$-algebra homomorphism $\rho^\mu : \mathcal{H}^\mu(q)\twoheadrightarrow {\rm TL}^{\mu}(q)$ defined in \S\ref{Young}. The map $\rho^{\mu}$ induces a surjective $\C[q,q^{-1}]$-algebra homomorphism
${\rm Mat}_{m_\mu}(\mathcal{H}^\mu(q))\twoheadrightarrow {\rm Mat}_{m_\mu}({\rm TL}^{\mu}(q))$, which we also denote by $\rho^\mu$. We obtain that
$$ \rho^{\mu}\circ \Psi_\mu : E_\mu \YH \rightarrow {\rm Mat}_{m_\mu}({\rm TL}^{\mu}(q))$$
  is a surjective $\C[q,q^{-1}]$-algebra homomorphism.

In order for $\rho^{\mu}\circ \Psi_\mu$ to factor through $E_\mu \YH/ E_\mu I_{d,n} \cong E_\mu \FTL$, all elements of $E_\mu I_{d,n}$ have to belong to the kernel of $\rho^{\mu}\circ \Psi_\mu$. Since $I_{d,n}$ is the ideal generated by the element
$e_1e_2 g_{1,2}$, it is enough to show that $(\rho^{\mu}\circ \Psi_\mu)(e_1e_2 g_{1,2}) =0$. This is immediate by Proposition 
\ref{imageofpsi}. Hence, if we denote by $\theta^{\mu}$ the natural surjection
 $E_\mu \YH \twoheadrightarrow E_\mu \YH/ E_\mu I_{d,n} \cong E_\mu \FTL$, there exists
 a unique $\C[q,q^{-1}]$-algebra homomorphism $\psi_\mu : E_\mu \FTL \rightarrow {\rm Mat}_{m_\mu}({\rm TL}^{\mu}(q))$ such that
 the following diagram is commutative:
 \begin{equation}\label{diag1} \,\,\,\, \,\,\,\, \,\,\,\, \,\,\,\, \,\,\,\,  \,\,\,\, \,\,\,\, 
 \diagram E_\mu \YH \dto^{\theta^{\mu}}&\rto^{\Psi_\mu}& &
          {\rm Mat}_{m_\mu}(\mathcal{H}^\mu(q)) \dto^{\rho^{\mu}}\\
            E_\mu \FTL & \rto^{\psi_\mu} & & {\rm Mat}_{m_\mu}({\rm TL}^{\mu}(q)) & & \enddiagram 
            \end{equation}
Since $\rho^{\mu}\circ \Psi_\mu$ is surjective, $\psi_\mu$ is also surjective.

\subsection{From Temperley--Lieb to $\FTL$} We now consider the surjective 
$\C[q,q^{-1}]$-algebra homomorphism:
$$ \theta^{\mu}\circ \Phi_\mu : {\rm Mat}_{m_\mu}(\mathcal{H}^\mu(q)) \rightarrow { E_\mu \FTL},$$
where $\Phi_\mu$ is the inverse of $\Psi_\mu$.
In order for $ \theta^{\mu}\circ \Phi_\mu$ to factor through ${\rm Mat}_{m_\mu}({\rm TL}^{\mu}(q))$, we have to show
that $G_{i,i+1} M_{k,l}$ belongs to the kernel of  $\theta^{\mu}\circ \Phi_\mu$
for all $i=1,\ldots,n-2$ such that $G_{i,i+1} \in \mathcal{H}^\mu(q)$ (that is, $\{i,i+1\} \subseteq J^{\mu}$)
and for all $k,l \in \{1,\ldots,m_\mu\}$.
Since $$G_{i,i+1} M_{k,l} = M_{k,1} G_{i,i+1} M_{1,1} M_{1,l}$$
and $\theta^{\mu}\circ \Phi_\mu$ is an homomorphism of $\C[q,q^{-1}]$-algebras, it is enough to show that
 $(\theta^{\mu}\circ \Phi_\mu) ( G_{i,i+1} M_{1,1}) =0$.
 
 Let $i=1,\ldots,n-2$ such that $G_{i,i+1} \in \mathcal{H}^\mu(q)$.
 By definition of $\Phi_\mu$, and since $\pi_{\mu,1}=1$, we have
 \begin{equation}\label{almost 2}
 \Phi_\mu(G_{i,i+1} M_{1,1}) =  \widetilde{\Phi}_\mu(U_\mu^{-1} G_{i,i+1} M_{1,1}U_\mu) =\widetilde{\Phi}_\mu(G_{i,i+1} M_{1,1}) =  E_{\chi_1^\mu}g_{i,i+1} E_{\chi_1^\mu}.
 \end{equation} 
Now, since $G_{i,i+1} \in \mathcal{H}^\mu(q)$, there exists $j \in \{1,\ldots,d\}$ such that
$\mu_j>2$ and $G_{i,i+1} \in  \mathcal{H}_{\mu_j}(q)$, that is,
$i \in \{ \mu_1+\cdots+\mu_{j-1}+1,\ldots, \mu_1+\cdots+\mu_{j-1}+\mu_j-2\}$. By definition of $\chi_1^\mu$, we have
$$ \chi_1^\mu(t_{ \mu_1+\cdots+\mu_{j-1}+1} ) = \cdots = \chi_1^\mu(t_{ \mu_1+\cdots+\mu_{j-1}+\mu_j} ) =\xi_j,$$
whence
$$ \chi_1^\mu(t_i) = \chi_1^\mu(t_{i+1})=\chi_1^\mu(t_{i+2}) = \xi_j.$$
Following \eqref{eiEchi}, we obtain
$$\Phi_\mu(G_{i,i+1} M_{1,1}) =   E_{\chi_1^\mu}g_{i,i+1} E_{\chi_1^\mu} = E_{\chi_1^\mu}e_i e_{i+1} g_{i,i+1} E_{\chi_1^\mu}.$$
Since $e_i e_{i+1} g_{i,i+1}  \in I_{d,n}$, we deduce that
$(\theta^{\mu}\circ \Phi_\mu) ( G_{i,i+1} M_{1,1}) =0$, as desired.

We conclude that there exists
 a unique $\C[q,q^{-1}]$-algebra homomorphism $\phi_\mu :  {\rm Mat}_{m_\mu}({\rm TL}^{\mu}(q)) \rightarrow E_\mu \FTL $ such that the following diagram is commutative:
 \begin{equation}\label{diag2} \,\,\,\, \,\,\,\, \,\,\,\, \,\,\,\, \,\,\,\,  \,\,\,\, \,\,\,\, 
 \diagram E_\mu \YH \dto^{\theta^{\mu}}& & \lto_{\Phi_\mu}&
          {\rm Mat}_{m_\mu}(\mathcal{H}^\mu(q)) \dto^{\rho^{\mu}}\\
            E_\mu \FTL & & \lto_{\phi_\mu}  & {\rm Mat}_{m_\mu}({\rm TL}^{\mu}(q)) & & \enddiagram 
            \end{equation}
Since $\theta^{\mu}\circ \Phi_\mu$ is surjective, $\phi_\mu$ is also surjective.

\subsection{An isomorphism theorem for the  Framisation of the Temperley--Lieb algebra $\FTL$}

We are now ready to prove the main result of this section.

\begin{thm} \label{thm-iso-CP}
Let $\mu \in {\rm Comp}_d(n)$.  The linear map $\psi_\mu$
is an isomorphism of $\C[q,q^{-1}]$-algebras with inverse map $\phi_\mu$. As a consequence, the map
$$\psi_n:= \bigoplus_{\mu \in {\rm Comp}_d(n)}  \psi_\mu :  \FTL \rightarrow \bigoplus_{\mu \in {\rm Comp}_d(n)}  {\rm Mat}_{m_\mu}({\rm TL}^{\mu}(q))$$
is also an isomorphism of $\C[q,q^{-1}]$-algebras, with inverse map
$$\phi_n:= \bigoplus_{\mu \in {\rm Comp}_d(n)}  \phi_\mu :   \bigoplus_{\mu \in {\rm Comp}_d(n)}  {\rm Mat}_{m_\mu}({\rm TL}^{\mu}(q)) \rightarrow  \FTL.$$
\end{thm}

\begin{proof}
Since the diagrams \eqref{diag1} and \eqref{diag2} are commutative, we have
$$ \rho^{\mu}\circ \Psi_\mu = \psi_\mu \circ \theta^{\mu} \quad \text{and} \quad \theta^{\mu}\circ \Phi_\mu = \phi_\mu \circ \rho^{\mu}.$$
This implies that
$$ \rho^{\mu}\circ \Psi_\mu \circ \Phi_\mu = \psi_\mu \circ \phi_\mu \circ \rho^{\mu}   \quad \text{and} \quad 
\theta^{\mu}\circ \Phi_\mu \circ \Psi_\mu = \phi_\mu \circ \psi_\mu \circ \theta^{\mu}.$$
By Theorem \ref{our iso}, $\Psi_\mu \circ \Phi_\mu = {\rm id}_{ {\rm Mat}_{m_\mu}(\mathcal{H}^\mu(q))}$
and $\Phi_\mu \circ \Psi_\mu = {\rm id}_{E_\mu \YH}$, whence
$$ \rho^{\mu} = \psi_\mu \circ \phi_\mu \circ \rho^{\mu}   \quad \text{and} \quad 
\theta^{\mu} = \phi_\mu \circ \psi_\mu \circ \theta^{\mu}.$$
Since the maps $\rho^{\mu}$ and $ \theta^{\mu}$ are surjective, we obtain
$$\psi_\mu \circ \phi_\mu = {\rm id}_{ {\rm Mat}_{m_\mu}({\rm TL}^\mu(q))} \quad \text{and} \quad 
\phi_\mu \circ \psi_\mu = {\rm id}_{E_\mu \FTL} ,$$
as desired.
\end{proof}

\subsection{A basis for the Framisation of the Temperley--Lieb algebra $\FTL$}\label{subs-basis}
Let $n \in \N$.
Let $\underline{i}=(i_1,\ldots,i_p)$ and $\underline{k}=(k_1, \ldots k_p)$ be two $p$-tuplets of non-negative integers, with $0 \leq p \leq n-1$.
We denote by $\mathfrak{H}_n$ the set of pairs $(\underline{i},\underline{k})$ such that
$$1\leq i_1 < i_2< \cdots < i_p \leq n-1 \,\,\,\,\,\text{and}\,\,\,\,\, i_j - k_j > 0\,\,\,\,\,\forall\,j=1,\ldots,p.$$
For $(\underline{i},\underline{k}) \in \mathfrak{H}_n$, we set
$$G_{\ul{i},\ul{k}}:=(G_{i_1}G_{i_1-1}\ldots G_{i_1-k_1})(G_{i_2}G_{i_2-1}\ldots G_{i_2-k_2})\ldots (G_{i_p}G_{i_p-1}\ldots G_{i_p-k_p})\in \mathcal{H}_n(q).$$ 
We take $G_{\emptyset,\emptyset}$ to be equal to $1$.
We have that the set
$$\mathcal{B}_{\mathcal{H}_n(q)} := \{G_{\ul{i},\ul{k}} \,|\, (\underline{i},\underline{k}) \in \mathfrak{H}_n\} = \{G_w\,|\, w \in \mathfrak{S}_n\}$$
is the standard basis of $\mathcal{H}_n(q)$ as a $\C[q,q^{-1}]$-module.

Now, let  us denote by $\mathfrak{T}_n$ the subset of $\mathfrak{H}_n$ consisting of the pairs $(\underline{i},\underline{k})$ such that
$$1\leq i_1 < i_2< \cdots < i_p \leq n-1 \,\,\,\,\,\text{and}\,\,\,\,\, 1\leq i_1-k_1 < i_2-k_2< \cdots < i_p-k_p \leq n-1.$$
Jones \cite{jo1} has shown that the set
$$\mathcal{B}_{{\rm TL}_n(q)} := \{G_{\ul{i},\ul{k}} \,|\, (\underline{i},\underline{k}) \in \mathfrak{T}_n\}$$
is a basis of ${\rm TL}_n(q)$ as a $\C[q,q^{-1}]$-module. We have $|\mathcal{B}_{{\rm TL}_n(q)} |=C_n$. 
By Theorem \ref{thm-iso-CP}, we obtain the following basis for $\FTL$:

\begin{prop}\label{FTLbasis}
The set
$$ \left\{  \phi_\mu (b_1b_2\ldots b_d\, M_{k,l}) \,|\, \mu \in {\rm Comp}_d(n), b_i \in \mathcal{B}_{{\rm TL}_{\mu_i}(q)}  \text{ for all }
i=1,\ldots,d, 1 \leq k,l \leq m_\mu     \right\}$$
is a basis of \,$\FTL$ as a $\C[q,q^{-1}]$-module. 
In particular, $\FTL$ is a free $\C[q,q^{-1}]$-module of rank 
$$\sum_{\mu \in {\rm Comp}_d(n)}m_\mu^2\, C_{\mu_1}C_{\mu_2} \cdots C_{\mu_d}.$$
\end{prop}

\begin{rem}{\rm Theorem \ref{res2} is a consequence of Proposition \ref{FTLbasis}, but we decided to keep the proof that uses the irreducible representations of $\C(q)\FTL$.}
\end{rem}

\section{Representation theory and an isomorphism theorem for the Complex Reflection Temperley--Lieb algebra}

In this section, we determine the irreducible representations and calculate the dimension of the Complex Reflection Temperley--Lieb algebra, which is also defined as a quotient of the Yokonuma--Hecke algebra of type $A$ \cite{gjkl2}. 
We then prove an isomorphism theorem similar to Theorem \ref{thm-iso-CP} and produce a basis for the Complex Reflection Temperley--Lieb algebra.
Our results here reinforce the opinion that the Framisation of the Temperley--Lieb algebra is the most natural analogue of the Temperley--Lieb algebra in this case.

\subsection{The Complex Reflection Temperley--Lieb algebra ${\rm CTL}_{d,n}(q)$} 
Let $j=1,\ldots,n$. We set
$$T_j := \frac{1}{d} \sum_{s=0}^{d-1}t_j^s .$$
We define the \emph{Complex Reflection Temperley--Lieb algebra} to be the quotient 
${\rm Y}_{d,n}(q)/\mathfrak{I}_{d,n}$, where $\mathfrak{I}_{d,n}$ is the ideal generated by the element
$T_1 e_1e_2\, g_{1,2}$  (if $n \leq 2$, we take $\mathfrak{I}_{d,n} = \{0\}$). 
The element $T_1$ commutes with $e_1e_2\, g_{1,2}$, since, for all $s=0,1,\ldots,d-1$, we have
$$ t_1^s e_1e_2\, g_{1,2} = t_1^s g_{1,2} e_1e_2 =
(t_1^s+ g_1t_2^s + g_2 t_1^s + g_1g_2t_3^s+g_2g_1t_2^s+ g_1g_2g_1t_3^s)e_1e_2$$
and
$t_1^s e_1 e_2 = t_2^s e_1 e_2 = t_3^s e_1 e_2 = e_1e_2 t_1^s\,$ due to Equation \eqref{eiti}. 
Further, we have
$ T_i  e_ie_{i+1}g_{i,i+1} \in \mathfrak{I}_{d,n}$  for all $i=1,\ldots,n-2$, since
$$T_i e_ie_{i+1}g_{i,i+1} = (g_1g_2 \ldots g_{n-1})^{i-1} \, T_1 e_1e_2 \,g_{1,2}\, (g_1g_2 \ldots g_{n-1})^{-(i-1)}.$$

\begin{rem}{\rm The ideal $\mathfrak{I}_{d,n}$ is the ideal generated by the element
$\sum_{0\leq a,b,c \leq d-1} t_1^a t_2^b t_3 ^c\,g_{1,2}$, which is the sum of all standard basis elements of ${\rm Y}_{d,3}(q)$.}
\end{rem}

\begin{rem}{\rm
For $d=1$, the Complex Reflection Temperley--Lieb algebra ${\rm CTL}_{1,n}(q)$ coincides with the classical Temperley--Lieb algebra $\mathrm{TL}_n(q)$.}
\end{rem}

\subsection{Irreducible representations of $\C(q){\rm CTL}_{d,n}(q)$}
Since the algebra $\C(q){\rm Y}_{d,n}(q)$ is semisimple, the algebra
$\C(q){\rm CTL}_{d,n}(q):=\C(q) \otimes_{\C[q,q^{-1}]}{\rm CTL}_{d,n}(q)$ is also semisimple. Moreover, we  have that the irreducible representations of $\C(q){\rm CTL}_{d,n}(q)$ are 
precisely the irreducible representations of $\C(q){\rm Y}_{d,n}(q)$ that pass to the quotient. That is, given $\blambda \in \mathcal{P}(d,n)$,
$V_{\blambda}$ is an irreducible representation of $\C(q){\rm CTL}_{d,n}(q)$ if and only if
$T_1e_1 e_2 g_{1,2} (\bv_{_{\cT}}) =0$ for every standard $d$-tableau ${\cT}$ of shape $\blambda$. 

Let $\blambda \in \mathcal{P}(d,n)$ and let $\cT$ be a standard $d$-tableau ${\cT}$ of shape $\blambda$.
Let $\{\xi_1,\dots,\xi_d\}$ be the set of all $d$-th roots of unity, ordered so that $\xi_1=1$. Recall that, following Theorem
\ref{thm-rep}, the action of the generators 
$t_1,t_2,\ldots,t_n$ on the basis element $\bv_{_{\cT}}$ of $V_{\blambda}$ is defined as follows:
$$t_j(\bv_{_{\cT}})=\xi_{\pos_j}\bv_{_{\cT}} \quad \text{ for } j=1,\dots,n,$$
where $\pos_j:=\pos(\cT|j)$. We deduce that, for $j=1,\dots,n$, we have
\begin{equation}\label{T_j}
T_j(\bv_{_{\cT}}) = \left\{
\begin{array}{ll}
 \bv_{_{\cT}} & \text{if }  \pos_{j} = 1 ; \\
 0 & \text{if }  \pos_{j} \neq 1.
 \end{array}\right.
\end{equation}

\begin{thm}\label{rep-ctl}
We have that $V_{\blambda}$ is an irreducible representation of $\C(q){\rm CTL}_{d,n}(q)$ if and only if 
the Young diagram of $\blambda^{(1)}$ has at most two columns.
\end{thm}

\begin{proof}
We have that 
$V_{\blambda}$ is an irreducible representation of $\C(q){\rm CTL}_{d,n}(q)$ if and only if
$T_1e_1 e_2 g_{1,2} (\bv_{_{\cT}}) =0$ for every standard $d$-tableau ${\cT}=(\cT^{(1)},\ldots,\cT^{(d)})$ of shape $\blambda$. 
Combining Equation \eqref{T_j} with Equation  (\ref{eirep}) yields:
$$T_1e_1 e_2 g_{1,2} (\bv_{_{\cT}}) =  g_{1,2}  e_1 e_2 T_1(\bv_{_{\cT}}) 
\left\{
\begin{array}{ll}
 g_{1,2} (\bv_{_{\cT}}) & \text{if }  \pos_{1} = \pos_{2}=\pos_{3}=1 ; \\
 0 & \text{otherwise. }  
 \end{array}\right.
$$

Now, if $ \pos_{1} = \pos_{2}=\pos_{3}=1$, then
$g_{1,2}$ acts on $\bv_{_{\cT}}$ in the same way that $G_{1,2}$ acts on $\rv_{_{{\cT}^{(1)}}}$ (replacing the entries greater than $3$ by entries in $\{4,\ldots,|\blambda^{(1)}|\}$).  
Following Proposition \ref{classical case},
we have $g_{1,2} (\bv_{_{\cT}}) = 0$ if and only if the Young diagram of $\blambda^{(1)}$ has at most two columns, as desired.
\end{proof}

\begin{rem}{\rm If the roots of unity $\{\xi_1,\dots,\xi_d\}$ are ordered so that $\xi_i=1$ for some $i>1$, then the above proposition holds for $\blambda^{(i)}$ in the place of $\blambda^{(1)}$.}
\end{rem}

\subsection{The dimension of $\C(q){\rm CTL}_{d,n}(q)$} We will now use the complete description of the irreducible representations of $\C(q){\rm CTL}_{d,n}(q)$ by
Theorem \ref{rep-ctl} to obtain a dimension formula for $\C(q){\rm CTL}_{d,n}(q)$. 

\begin{thm} We have
$${\rm dim}_{\C(q)}(\C(q){\rm CTL}_{d,n}(q)) = \sum_{k=0}^n  \binom{n}{k}^2  C_k\, (d-1)^{n-k} (n-k)! .$$
\end{thm}

\begin{proof}
Let us denote by $\mathcal{P}^{\leq 2}(n)$ the set of partitions of of $n$ whose Young diagram  has at most two columns. By Theorem \ref{rep-ctl}, and since the algebra $\C(q){\rm CTL}_{d,n}(q)$ is semisimple, 
we have
$${\rm dim}_{\C(q)}(\C(q){\rm CTL}_{d,n}(q)) =
\sum_{ \begin{array}{c}\scriptstyle{\blambda \in \mathcal{P}(d,n)}\\
 \scriptstyle{\blambda^{(1)} \in \mathcal{P}^{\leq 2}(| \blambda^{(1)} |)}\end{array}}
\mathrm{dim}_{\C(q)} (V_{\blambda})^2,$$
where $\mathrm{dim}_{\C(q)} (V_{\blambda})$ is the number of standard $d$-tableaux of shape $\blambda$.

Let $k \in \{0,1,\ldots,n\}$. Let
$\blambda \in \mathcal{P}(d,n)$ be such that $\blambda^{(1)} \in \mathcal{P}^{\leq 2}(k)$, where $k = | \blambda^{(1)} |$.
Since there are $\binom{n}{k}$ ways to choose the numbers in $\{1,\ldots,n\}$ that will be placed in the nodes of the Young diagram of $\blambda^{(1)}$, we have
$$  \mathrm{dim}_{\C(q)} (V_{\blambda}) = \binom{n}{k} \, \mathrm{dim}_{\C(q)} (V_{\blambda^{(1)}}) \,
 \mathrm{dim}_{\C(q)} (V_{(\blambda^{(2)},\ldots, \blambda^{(d)})}),$$ 
where $V_{\blambda^{(1)}}$ is the irreducible representation of $\C(q){\rm TL}_{k}(q)$ labelled by $\blambda^{(1)}$ and
$V_{(\blambda^{(2)},\ldots, \blambda^{(d)})}$ is the irreducible representation of $\C(q){\rm Y}_{d-1,n-k}(q)$ labelled by
 $(\blambda^{(2)},\ldots, \blambda^{(d)}) \in \mathcal{P}(d-1,n-k)$.
 We deduce that
 $$  {\rm dim}_{\C(q)}(\C(q){\rm CTL}_{d,n}(q)) = \sum_{k=0}^n \binom{n}{k}^2  \left(
\sum_{ \begin{array}{c}\scriptstyle{\lambda \in  \mathcal{P}^{\leq 2}(k)}\\
 \scriptstyle{\bmu \in \mathcal{P}(d-1,n-k)}\end{array}}
  \mathrm{dim}_{\C(q)} (V_{\lambda})^2 \,
 \mathrm{dim}_{\C(q)} (V_{\bmu})^2\right).
 $$
 Now, we have that the sum inside the parenthesis is equal to
 $$ \left( \sum_{\lambda \in  \mathcal{P}^{\leq 2}(k)}
  \mathrm{dim}_{\C(q)} (V_{\lambda})^2 \right) \left(
  \sum_{\bmu \in \mathcal{P}(d-1,n-k)}
 \mathrm{dim}_{\C(q)} (V_{\bmu})^2 \right) = C_k\, (d-1)^{n-k} (n-k)! \, ,$$
 whence we obtain the desired result.
\end{proof}

\begin{rem}{\rm
Note that the dimension of ${\rm CTL}_{d,n}(q)$ can be rewritten, using the set 
${\rm Comp}_d(n)$, in the following way:
$${\rm dim}_{\C(q)}(\C(q){\rm CTL}_{d,n}(q)) = \sum_{\mu \in {\rm Comp}_d(n)} \left( \frac{n!}{\mu_1!\mu_2!\ldots \mu_d!}\right)^2 C_{\mu_1} \, \mu_2!   \ldots \mu_d! ,$$
using the fact that
$$ \sum_{\mu \in {\rm Comp}_d(n)}  \frac{n!}{\mu_1!\mu_2!\ldots \mu_d!} = d^n.$$
}
\end{rem}

\subsection{An isomorphism theorem for the  Complex Reflection Temperley--Lieb algebra ${\rm CTL}_{d,n}(q)$}
We  will now reuse the notation of Section \ref{sec-iso}, and assume again, without loss of generality, that $\xi_1=1$.
For all $j=1,\ldots,n$ and $\chi \in \Irr(\mathcal{A}_{d,n})$, we have
\begin{equation}\label{TjEchi}
T_j E_\chi = E_\chi T_j= \frac{1}{d}\sum_{s=0}^{d-1} \chi(t_j)^s E_\chi =  
\left\{ \begin{array}{cl}
E_\chi & \text{ if }  \chi(t_j) = 1 ; \\ & \\
0  & \text{ if }  \chi(t_j) \neq 1.
\end{array}\right.
\end{equation}

\begin{prop}\label{imageofpsi2}
Let $\mu \in {\rm Comp}_d(n)$ and $k \in \{1,\ldots,m_\mu\}$. We have
$$
\Psi_\mu(E_{\chi_k^\mu}T_1e_1e_2g_{1,2}) = 
\left\{ \begin{array}{cl}
G_{1,2} \,M_{k,k} & \text{ if }  \chi_k^\mu(t_1)=\chi_k^\mu(t_2)=\chi_k^\mu(t_3)=1; \\ & \\
0  & \text{ otherwise } .
\end{array}\right.
$$
Thus, $\Psi_\mu(E_\mu T_1 e_1e_2g_{1,2})$ is a diagonal matrix in $ {\rm Mat}_{m_\mu}(\mathcal{H}^\mu(q))$
with all non-zero coefficients being equal to $G_{1,2}$.
\end{prop}

\begin{proof}
Combining \eqref{TjEchi} with \eqref{i-1} yields:
\begin{equation}\label{Tjplus}
E_{\chi_k^\mu}T_1e_1e_2g_{1,2} = 
\left\{ \begin{array}{cl}
E_{\chi_k^\mu}g_{1,2} & \text{ if }  \chi_k^\mu(t_1)=\chi_k^\mu(t_2)=\chi_k^\mu(t_3)=1; \\ & \\
0  & \text{ otherwise } .
\end{array}\right.
\end{equation}
If now $\chi_k^\mu(t_1)=\chi_k^\mu(t_2)=\chi_k^\mu(t_3)=1$, then the following hold: \smallbreak
\begin{itemize}
\item $\mu_1 > 2$\,; \smallbreak
\item $w(\chi_k^\mu)=\chi_k^\mu$ for all $w \in {\langle s_1,s_2\rangle} \subseteq \mathfrak{S}_n$\,; \smallbreak
\item there exists
$\sigma \in {\langle s_4,\ldots,s_{n-1} \rangle} \subseteq \mathfrak{S}_n$ such that
$\chi_k^{\mu} = \sigma(\chi_1^{\mu})$, since
$\chi_1^\mu(t_1)=\chi_1^\mu(t_2)=\chi_1^\mu(t_3)=1$. 
\end{itemize}
We have $\sigma \mathfrak{S}_\mu = \pi_{\mu,k}  \mathfrak{S}_\mu$, and so there exists $x \in \mathfrak{S}_\mu$ such that
$\sigma=\pi_{\mu, k}x$ and $\ell(\sigma)=\ell(\pi_{\mu,k})+\ell(x)$. Thus, the product of a reduced expression for $\pi_{\mu,k}$ and a reduced expression for $x$ yields a reduced expression for $\sigma$. However, all reduced expressions for $\sigma$ have all their
factors in $\{ s_4,\ldots,s_{n-1} \}$  (see, for example, \cite[Proposition 1.2.10]{GePf}). 
Therefore, $\pi_{\mu,k} \in {\langle s_4,\ldots,s_{n-1} \rangle}$, whence we obtain that $\pi_{\mu,k}w=w\pi_{\mu,k}$ for all $w \in {\langle s_1,s_2\rangle} $.
Thus, similarly to \eqref{almost 1}, we have
$$ \Psi_\mu(E_{\chi_k^\mu}g_{1,2}) = \sum_{w \in {\langle s_1,s_2\rangle}}\Psi_\mu(E_{\chi_k^\mu}g_w) = 
\sum_{w \in{\langle s_1,s_2\rangle}} G_{\pi_{\mu,k}^{-1}w\pi_{\mu,k}} M_{k,k} =
\sum_{w \in{\langle s_1,s_2\rangle}} G_{w} M_{k,k} = G_{1,2}\, M_{k,k}. $$
Combining this with \eqref{Tjplus} yields the desired result.
\end{proof}

\begin{exmp}{\rm Let us consider again the case $d=2$ and $n=4$ as in Example \ref{example1}. We have 
$$
\Psi_\mu(E_{\chi_k^\mu}T_1e_1e_2g_{1,2}) = 
\left\{ \begin{array}{cl}
G_{1,2} & \text{ if $\mu=(4,0)$ }, \\ & \\
G_{1,2}\,M_{1,1}  & \text{ if $\mu=(3,1)$ and $k=1$ }, \\  & \\
0  & \text{ otherwise } .
\end{array}\right.$$}
\end{exmp}

We now consider the natural surjection $\rho_1 : \mathcal{H}_{\mu_1}(q) \twoheadrightarrow \mathcal{H}_{\mu_1}(q)/I_{\mu_1} \cong {\rm TL}_{\mu_1}(q)$, where $I_{\mu_1}$ is the ideal generated by $G_{1,2}$ if $\mu_1 >2$ and 
$I_{\mu_1} = \{0\}$ if $\mu_1 \leq 2$. 
Let us denote by $\mathcal{H}^{(\mu_2,\ldots,\mu_d)}(q)$ the tensor
product of Iwahori--Hecke algebras $\mathcal{H}_{\mu_2}(q) \otimes \cdots \otimes \mathcal{H}_{\mu_d}(q)$. 
Then the map
${\rm P}_{1}:= \rho_1 \otimes {\rm id}_{\mathcal{H}^{(\mu_2,\ldots,\mu_d)}(q)} $
is a surjective $\C[q,q^{-1}]$-algebra homomorphism 
$\mathcal{H}^\mu(q)\twoheadrightarrow {\rm TL}_{\mu_1}(q) \otimes 
\mathcal{H}^{(\mu_2,\ldots,\mu_d)}(q)$.
This in turn induces a surjective $\C[q,q^{-1}]$-algebra homomorphism
${\rm Mat}_{m_\mu}(\mathcal{H}^\mu(q))\twoheadrightarrow {\rm Mat}_{m_\mu}({\rm TL}_{\mu_1}(q) \otimes \mathcal{H}^{(\mu_2,\ldots,\mu_d)}(q))$, which we also denote by ${\rm P}_{1}$. 
We obtain that
$$ {\rm P}_{1}\circ \Psi_\mu : E_\mu \YH \rightarrow {\rm Mat}_{m_\mu}({\rm TL}_{\mu_1}(q) \otimes  \mathcal{H}^{(\mu_2,\ldots,\mu_d)}(q))$$
  is a surjective $\C[q,q^{-1}]$-algebra homomorphism.

In order for ${\rm P}_{1}\circ \Psi_\mu$ to factor through $E_\mu \YH/ E_\mu \mathfrak{I}_{d,n} \cong E_\mu {\rm CTL}_{d,n}(q)$, all elements of $E_\mu \mathfrak{I}_{d,n}$ have to belong to the kernel of ${\rm P}_{1}\circ \Psi_\mu$. Since $\mathfrak{I}_{d,n}$ is the ideal generated by the element
$T_1e_1e_2 g_{1,2}$, it is enough to show that $({\rm P}_{1}\circ \Psi_\mu)(T_1e_1e_2 g_{1,2}) =0$. This is immediate by Proposition 
\ref{imageofpsi2}. Hence, if we denote by $\Theta^{\mu}$ the natural surjection
 $E_\mu \YH \twoheadrightarrow E_\mu \YH/ E_\mu \mathfrak{I}_{d,n} \cong E_\mu{\rm CTL}_{d,n}(q)$, there exists
 a unique $\C[q,q^{-1}]$-algebra homomorphism $\overline{\psi}_\mu : E_\mu{\rm CTL}_{d,n}(q) \rightarrow {\rm Mat}_{m_\mu}({\rm TL}_{\mu_1}(q) \otimes  \mathcal{H}^{(\mu_2,\ldots,\mu_d)}(q))$ such that
 the following diagram is commutative:
 \begin{equation}\label{diag3} \,\,\,\, \,\,\,\, \,\,\,\, \,\,\,\, \,\,\,\,  \,\,\,\, \,\,\,\, 
 \diagram E_\mu \YH \dto^{\Theta^{\mu}}&\rto^{\Psi_\mu}& &
          {\rm Mat}_{m_\mu}(\mathcal{H}^\mu(q)) \dto^{{\rm P}_{1}}\\
            E_\mu {\rm CTL}_{d,n}(q) & \rto^{\overline{\psi}_\mu} & & {\rm Mat}_{m_\mu}({\rm TL}_{\mu_1}(q) \otimes  \mathcal{H}^{(\mu_2,\ldots,\mu_d)}(q))& & \enddiagram 
            \end{equation}
Since ${\rm P}_{1}\circ \Psi_\mu$ is surjective, $\overline{\psi}_\mu$ is also surjective.

Let us now consider the surjective 
$\C[q,q^{-1}]$-algebra homomorphism:
$$ \Theta^{\mu}\circ \Phi_\mu : {\rm Mat}_{m_\mu}(\mathcal{H}^\mu(q)) \rightarrow { E_\mu {\rm CTL}_{d,n}(q)},$$
where $\Phi_\mu$ is the inverse of $\Psi_\mu$.
In order for $ \Theta^{\mu}\circ \Phi_\mu$ to factor through ${\rm Mat}_{m_\mu}({\rm TL}_{\mu_1}(q) \otimes  \mathcal{H}^{(\mu_2,\ldots,\mu_d)}(q))$, we have to show
that $G_{1,2} M_{k,l}$ belongs to the kernel of  $\Theta^{\mu}\circ \Phi_\mu$
if $\mu_1 >2$.
Since $$G_{1,2} M_{k,l} = M_{k,1} G_{1,2} M_{1,1} M_{1,l}$$
and $\Theta^{\mu}\circ \Phi_\mu$ is an homomorphism of $\C[q,q^{-1}]$-algebras, it is enough to show that
 $(\Theta^{\mu}\circ \Phi_\mu) ( G_{1,2} M_{1,1}) =0$.
 
 Similarly to \eqref{almost 2}, by definition of $\Phi_\mu$, and since $\pi_{\mu,1}=1$, we have
 $$\Phi_\mu(G_{1,2} M_{1,1}) =   E_{\chi_1^\mu}g_{1,2} E_{\chi_1^\mu}.$$ 
Now, by definition of $\chi_1^\mu$, we have
$$ \chi_1^\mu(t_{1} ) = \cdots = \chi_1^\mu(t_{ \mu_1} ) =\xi_1=1.$$
Since $\mu_1>2$, we deduce that
$$ \chi_1^\mu(t_1) = \chi_1^\mu(t_{2})=\chi_1^\mu(t_{3}) = 1.$$
Following \eqref{Tjplus}, we obtain
$$\Phi_\mu(G_{1,2} M_{1,1}) =   E_{\chi_1^\mu}g_{1,2} E_{\chi_1^\mu} = E_{\chi_1^\mu}T_1e_1 e_{2} g_{1,2} E_{\chi_1^\mu}.$$
Since $T_1 e_1 e_{2} g_{1,2}  \in \mathfrak{I}_{d,n}$, we deduce that
$(\Theta^{\mu}\circ \Phi_\mu) ( G_{1,2}\, M_{1,1}) =0$, as desired.

We conclude that there exists
 a unique $\C[q,q^{-1}]$-algebra homomorphism $$\overline{\phi}_\mu :  {\rm Mat}_{m_\mu}({\rm TL}_{\mu_1}(q) \otimes  \mathcal{H}^{(\mu_2,\ldots,\mu_d)}(q)) \rightarrow E_\mu {\rm CTL}_{d,n}(q) $$ such that the following diagram is commutative:
 \begin{equation}\label{diag4} \,\,\,\, \,\,\,\, \,\,\,\, \,\,\,\, \,\,\,\,  \,\,\,\, \,\,\,\, 
 \diagram E_\mu \YH \dto^{\Theta^{\mu}}& & \lto_{\Phi_\mu}&
          {\rm Mat}_{m_\mu}(\mathcal{H}^\mu(q)) \dto^{{\rm P}_{1}}\\
            E_\mu{\rm CTL}_{d,n}(q)  & & \lto_{\overline{\phi}_\mu}  & {\rm Mat}_{m_\mu}({\rm TL}_{\mu_1}(q) \otimes  \mathcal{H}^{(\mu_2,\ldots,\mu_d)}(q)) & & \enddiagram 
            \end{equation}
Since $\Theta^{\mu}\circ \Phi_\mu$ is surjective, $\overline{\phi}_\mu$ is also surjective.

Similarly to Theorem \ref{thm-iso-CP}, we obtain the following result:

\begin{thm} \label{thm-iso-CP2}
Let $\mu \in {\rm Comp}_d(n)$.  The linear map $\overline{\psi}_\mu$
is an isomorphism of $\C[q,q^{-1}]$-algebras with inverse map $\overline{\phi}_\mu$. As a consequence, the map
$$\overline{\psi}_n:= \bigoplus_{\mu \in {\rm Comp}_d(n)}  \overline{\psi}_\mu :  {\rm CTL}_{d,n}(q) \rightarrow \bigoplus_{\mu \in {\rm Comp}_d(n)}  {\rm Mat}_{m_\mu}({\rm TL}_{\mu_1}(q) \otimes  \mathcal{H}^{(\mu_2,\ldots,\mu_d)}(q))$$
is also an isomorphism of $\C[q,q^{-1}]$-algebras, with inverse map
$$\overline{\phi}_n:= \bigoplus_{\mu \in {\rm Comp}_d(n)}  \overline{\phi}_\mu :   \bigoplus_{\mu \in {\rm Comp}_d(n)}  {\rm Mat}_{m_\mu}({\rm TL}_{\mu_1}(q) \otimes  \mathcal{H}^{(\mu_2,\ldots,\mu_d)}(q)) \rightarrow  {\rm CTL}_{d,n}(q).$$
\end{thm}

\subsection{A basis for the Complex ReflectionTemperley--Lieb algebra $ {\rm CTL}_{d,n}(q)$}

By Theorem \ref{thm-iso-CP2}, using the notation of \S \ref{subs-basis}, we obtain the following basis for ${\rm CTL}_{d,n}(q)$:

\begin{prop}\label{FTLbasis}
The set
$$ \left\{  \overline{\phi}_\mu (b_1b_2\ldots b_d\, M_{k,l}) \,|\, \mu \in {\rm Comp}_d(n), b_1 \in \mathcal{B}_{{\rm TL}_{\mu_1}(q)},
b_i \in \mathcal{B}_{\mathcal{H}_{\mu_i}(q)}
  \text{ for all }
i=2,\ldots,d, 1 \leq k,l \leq m_\mu     \right\}$$
is a basis of  ${\rm CTL}_{d,n}(q)$ as a $\C[q,q^{-1}]$-module. 
In particular, ${\rm CTL}_{d,n}(q)$ is a free $\C[q,q^{-1}]$-module of rank 
$$\sum_{\mu \in {\rm Comp}_d(n)} m_\mu^2\, C_{\mu_1} {\mu_2}! \ldots {\mu_d}!.$$
\end{prop}

\end{document}